\documentclass[reqno,11pt]{amsart}
\usepackage{amsthm,amsfonts,amssymb,euscript,
mathrsfs,graphics,color,amsmath,latexsym,marginnote}
\usepackage{dsfont}  
 
\theoremstyle{plain}

\usepackage{hyperref}
\usepackage{times}

\numberwithin{equation}{section}

\newcommand{\U}{\underline{U}}

\renewcommand{\H}{{\cal H}}

\newcommand{\F}{\mathcal{F}}

\newtheorem{theorem}{Theorem}[section]
\newtheorem{proposition}[theorem]{Proposition}
\newtheorem{lemma}[theorem]{Lemma}
\newtheorem{corollary}[theorem]{Corollary}

\newtheorem{remark}[theorem]{Remark}
\newtheorem{remarks}[theorem]{Remark}

\newtheorem{definition}[theorem]{Definition}
\newcommand{\be}{\begin{equation}}
\newcommand{\ee}{\end{equation}}

\newcommand{\uno}{\mathds{1}}
\newcommand{\e}{\varepsilon}

\newcommand{\ov}{\overline}

\newcommand{\R}{\mathbb R}
\newcommand{\C}{\mathbb C}

\newcommand{\Z}{\mathbb Z}

\newcommand{\N}{\mathbb N}
\newcommand{\T}{\mathbb T}

\newcommand{\s }{\sigma }
\newcommand{\ii }{{\rm i} }

\renewcommand{\o }{\omega }
\renewcommand{\O }{\mathcal{O} }

\newcommand{\x }{\xi }

\newcommand{\diag}{{\rm diag}}

\newcommand{\A}{\mathcal{A}}

\newcommand{\opw}{{Op^{\mathrm{W}}}}
\newcommand{\ophw}{{Op_h^{\mathrm{W}}}}
\newcommand{\opbw}{{Op^{\mathrm{BW}}}}

\def\supp#1{{\rm supp(#1)}}

\def\hat{\widehat}

\def\bar{\overline}

\def\cal{\mathcal}

\def\poscalr#1#2{\langle#1,#2\rangle}
\def\poscals#1#2{\langle#1,#2\rangle_{\mathcal{H}^0}}
\renewcommand{\Re}{\mathrm{Re}\,}
\renewcommand{\Im}{\mathrm{Im}\,}

\def\ba{\begin{aligned}}
\def\ea{\end{aligned}}
\def\beginm{\begin{multline}}
\def\endm{\end{multline}}

\providecommand{\vect}[2]{{\bigl[\begin{smallmatrix}#1\\#2\end{smallmatrix}\bigr]}}   
\providecommand{\sm}[4]{{\bigl[\begin{smallmatrix}#1&#2\\#3&#4\end{smallmatrix}\bigr]}}

\makeatletter

\def\l@subsection{\@tocline{2}{0pt}{2.5pc}{5pc}{}}
\def\l@subsubsection{\@tocline{3}{0pt}{4.5pc}{5pc}{}}
\renewcommand\tocchapter[3]{%
  \indentlabel{\@ifnotempty{#2}{\ignorespaces#2.\quad}}#3%
}
\newcommand\@dotsep{4.5}
\def\@tocline#1#2#3#4#5#6#7{\relax
  \ifnum #1>\c@tocdepth 
  \else
    \par \addpenalty\@secpenalty\addvspace{#2}%
    \begingroup \hyphenpenalty\@M
    \@ifempty{#4}{%
      \@tempdima\csname r@tocindent\number#1\endcsname\relax
    }{%
      \@tempdima#4\relax
    }%
    \parindent\z@ \leftskip#3\relax \advance\leftskip\@tempdima\relax
    \rightskip\@pnumwidth plus1em \parfillskip-\@pnumwidth
    #5\leavevmode\hskip-\@tempdima{#6}\nobreak
    \leaders\hbox{$\m@th\mkern \@dotsep mu\hbox{.}\mkern \@dotsep mu$}\hfill
    \nobreak
    \hbox to\@pnumwidth{\@tocpagenum{#7}}\par
    \nobreak
    \endgroup
  \fi}
\def\l@subsection{\@tocline{2}{0pt}{2.5pc}{5pc}{}}

\begin{document}
\bibliographystyle{plain}

\title[Control  for quasilinear NLS]{Controllability of  quasi-linear Hamiltonian Schr\"odinger equations on tori}

\date{}

\author{Felice Iandoli} 
\address{\scriptsize{Universit\`a della Calabria}}
\email{felice.iandoli@unical.it}

\author{Jingrui Niu}
\address{\scriptsize{LJLL (Sorbonne Universit\'e)}}
\email{jingrui.niu@sorbonne-universite.fr}



\thanks{This research has been supported  by ERC grant ANADEL 757996: Analysis of the geometrical effects on dispersive equations. \\The first author has been partially supported 
by PRIN project 2017JPCAPN (Italy): Qualitative and quantitative aspects of nonlinear PDEs}

\begin{abstract}  
We prove exact controllability for quasi-linear Hamiltonian Schr\"odinger equations on tori of dimension greater or equal then two. The  result holds true for sufficiently small initial conditions satisfying natural minimal regularity assumptions, provided that the region of control satisfies the geometric control condition.
\end{abstract}  


\maketitle
\tableofcontents

\section{Introduction}
In this paper, we study the exact controllability for the following quasi-linear  perturbation of the Schr\"odinger equation 
\begin{equation}\label{NLS}
\begin{aligned}
&\ii u_{t}+\Delta u+g_1'(|u|^2)\Delta\big(g_1(|u|^2)\big)u+g_2(|u|^2)u=0 \,,\\
 \end{aligned}\end{equation}
 where $u=u(t,x)$, 
$ x=(x_1,\ldots,x_{d})\in \mathbb{T}^{d}:=(\mathbb{R}/2\pi \mathbb{Z})^{d}$ and $g_1$ and $g_2$ are polynomial functions of degree at least one vanishing at the origin.
In order to state our main theorem we need the following definition.
\begin{definition}\label{gcc}
We say that a nonempty open subset $\omega\subset \mathbb{T}^d$ satisfies the \textit{geometric control condition} if every geodesic of $\mathbb{T}^d$ eventually enters into $\omega$. In particular, since we endow $\mathbb{T}^d$ with the standard metric, this means that for every $(x,\xi)\in\mathbb{T}^d\times\mathbb{S}^{d-1}$, there exists some $t\in (0,\infty)$, such that $x+t\xi\in\omega$.
\end{definition}
\noindent We set $\langle j \rangle:=\sqrt{1+|j|^{2}}$ for $j\in \mathbb{Z}^{d}$,
we endow the classical Sobolev space $H^{s}(\mathbb{T}^{d};\mathbb{C})$ with the norm 
$
\|u(\cdot)\|_{H^{s}}^{2}:=\sum_{j\in \mathbb{Z}^{d}}\langle j\rangle^{2s}|u_{j}|^{2}\,
$, where $u_j$ is the corresponding Fourier coefficient of $u$.
The main result  is the following.
\begin{theorem}{\bf (Exact Controllability).}\label{main}
Let $T>0$, and suppose that  $\omega\subset \mathbb{T}^d$ satisfies the geometric control condition. For any $s>d/2+2$ there exists $\epsilon_0>0$ sufficiently small such that for all $u_{in},u_{end}\in H^{s}(\mathbb{T}^d;\mathbb{C})$ satisfying 
$    \|u_{in}\|_{H^{s}}+\|u_{end}\|_{H^{s}}<\epsilon_0$ the following holds.
There exists $f\in C([0,T],H^{s}(\mathbb{T}^d;\mathbb{C}))$ such that for all $supp\,f(t,\cdot)\subset\omega$  for any $t\in[0,T]$ and
     there exists a unique solution $u\in C([0,T],H^{s}(\mathbb{T}^d;\mathbb{C}))$ of
\begin{equation}\label{eq: scalar NLS with control}
\left\{\begin{array}{l}
      \ii u_{t}+\Delta u+g_1'(|u|^2)\Delta\big(g_1(|u|^2)\big)u+g_2(|u|^2)u=f,\quad (t,x)\in [0,T]\times\mathbb{T}^d ,   \\
     u|_{t=0}=u_{in},
\end{array}
\right.
\end{equation}
which verifies that $u|_{t=T}=u_{end}$. 
\end{theorem}
Theorem \ref{main} generalizes the one dimensional result by Baldi-Haus-Montalto in \cite{BHM} to the case of quasi-linear NLS in  dimension $d\geq 2$. The equation considered in \cite{BHM} is the most general (local) Hamiltonian NLS, we decided to study the less general case \eqref{NLS} in order to avoid extra technical  difficulties. 
Thanks to the special structure of the nonlinearity in \eqref{NLS}, we tried to optimize the amount of regularity required on the initial condition. For the equation in the full generality, i.e. as considered for instance in \cite{BHM} and \cite{FIJMPA, FIAP}, we do not know if it is possible to reach the control with these mild regularity assumptions: even  the problem of  the local well-posedness with mild regularity assumptions seems non-trivial. 

Besides these mathematical considerations, quasi-linear Schr\"odinger equations of the specific form \eqref{NLS} 
appear in many domains 
of physics, like plasma physics and fluid mechanics \cite{sergio,goldman}, 
quantum mechanics \cite{hasse}, condensed matter theory \cite{fedayn}. 
They are also important  in the study of Kelvin waves 
in the superfluid turbulence \cite{LLNR}. Many mathematicians studied these equations for local and long time existence: for a complete overview of the related literature we refer to \cite{FGI}.

The literature concerning linear and semi-linear (i.e. when $g_1=0$ in \eqref{NLS})  Schr\"odinger equations is wide. Without trying to be exhaustive we  quote Jaffard \cite{Jaff}, Lebeau \cite{Lebeau}, Beauchard-Laurent \cite{BL}, Dehman-G\'erard-Lebeau \cite{DGL}, Anantharaman-Maci\`a \cite{AM}. A complete overview of the results regarding linear and semi-linear Schr\"odinger equations may be found in the survey articles by Laurent \cite{Laurent} and Zuazua \cite{ZZ}.

The main difference between \eqref{NLS} and linear, or semi-linear, Schr\"odinger equations is the presence of two derivatives in the nonlinearity. When one linearizes a quasi-linear equation, one ends up with variable coefficients equation, for which Sobolev energy estimates are non-trivial.
In \cite{BHM}, the authors exploit the fact that they are working in a one dimensional context: they can completely remove the dependence on space of the coefficients of the linearized equation through a diffeomorphism of the circle. 
Hence, after a reparametrization of time, they essentially are in a semi-linear framework, therefore they can obtain observability for the linearized equation by means of Ingham inequalities. In such a way one controls linear problems. After that, a Nash-Moser iterative scheme is used to find the nonlinear control for the quasi-linear equation. One could avoid using a Nash-Moser scheme and rely on para-differential  approximation schemes {\em\`a la Kato} (see for instance \eqref{n-prob}). This is the strategy adopted by Alazard-Baldi-Kwan  in \cite{ABK} for the $2d$ gravity-capillary water waves equations, the torus diffeomorphism  is still performed in order to reduce the paralinearized equations to constant coefficients ones.

When working in dimension $d\geq 2$ it is not possible to remove the $x$-dependence on the coefficients of the equations, therefore a different strategy has to be adopted.  The breakthrough result in this direction is the one by Zhu \cite{Zhu} for the gravity-capillary water waves equation. Zhu uses an iterative scheme {\em\`a la Kato} and obtains {\em a priori} estimates by means of a modified energy, which is equivalent to the $H^s$-norm for $s$ sufficiently large. To obtain observability of linearized equations, since Ingham inequalities are not available in dimension $d\geq 2$, he uses a semi-classical approach (Lebeau \cite{Lebeau96}) for high frequencies and a uniqueness-compactness for low frequencies (Bardos-Lebeau-Rauch \cite{BLR}). Then one passes to the limit thanks to the {\em a priori} estimates.

We follow the lines of the aforementioned strategy introduced by Zhu, adapting it to \eqref{NLS} and trying to optimize  the amount of regularity needed on the initial condition. We now explain the proof of Theorem \ref{main} and the structure of the paper.

The general strategy to establish the exact controllability is to reduce it to the \textit{null controllability}, i.e. the exact controllability with null final datum ($u_{end}=0$). To be more specific, we first prove the following result.
\begin{theorem}{\bf (Null Controllability).}\label{main-null}
Let $T>0$, and suppose that  $\omega\subset \mathbb{T}^d$ satisfies the geometric control condition. For any $s>d/2+2$ there exists $\epsilon_0>0$ sufficiently small such that for all $u_{in}\in H^{s}(\mathbb{T}^d;\mathbb{C})$ satisfying 
$    \|u_{in}\|_{H^{s}}<\epsilon_0$ the following holds.
There exists $\Tilde{f}\in C([0,T],H^{s}(\mathbb{T}^d;\mathbb{C}))$ and
     there exists a unique solution $u\in C([0,T],H^{s}(\mathbb{T}^d;\mathbb{C}))$ of
\begin{equation}\label{eq: scalar NLS with null control}
\left\{\begin{array}{l}
      \ii u_{t}+\Delta u+g_1'(|u|^2)\Delta\big(g_1(|u|^2)\big)u+g_2(|u|^2)u=\chi_T\varphi_{\omega}\Tilde{f},\quad (t,x)\in [0,T]\times\mathbb{T}^d ,   \\
     u|_{t=0}=u_{in},
\end{array}
\right.
\end{equation}
which verifies that $u|_{t=T}=0$.  Furthermore, the control function verifies
\begin{enumerate}
    \item $\Tilde{f}\in C([0,T],H^{s}(\mathbb{T}^d;\mathbb{C}))$;
    \item $\chi_T(\cdot)=\chi_1(\cdot/T)\in C^{\infty}(\R)$ where $\chi_1(t)=1$ for $t\leq\frac{1}{2}$ and $\chi_1(t)=0$ for $t\geq\frac{3}{4}$;
    \item $0\leq \varphi_{\omega}\in C^{\infty}(\mathbb{T}^d)$ satisfies $\uno_{\omega'}\leq\varphi_{\omega}\leq\uno_{\omega}$, where $\omega'$ also satisfies the geometric control condition and $\Bar{\omega'}\subset\omega$. Such $\omega'$ exists because that $\T^d$ is compact.
    \end{enumerate}
\end{theorem}
As said, in order to prove Theorem \ref{main-null}, the first step is to consider a sequence of linear problems that approximates the non-linear one. In order to do that,  we  first paralinearize the equation in the sense of Bony. We obtain the system \eqref{NLSpara}, where $\mathcal{A}$ is a self-adjoint para-differential operator and $R$ is a bounded remainder. Once achieved the paralinearization we consider the sequence of problems \eqref{n-prob}, this is done in the spirit of \cite{ABK, FIAP, FIJMPA}.   The main difficulty is to prove the $L^2$- controllability of problem \eqref{n-prob}.
For simplicity, we may suppose that the smoothing remainder is equal to zero. In other words we consider  \begin{equation}\label{esempio}
\partial_tV=\ii E\mathcal{A}(\underline{U})V -\ii\chi_T\varphi_{\omega}EF,
\end{equation}
where $E$, $F$ are defined in \eqref{matriciozze}, $V:=(v,\bar v)$ and $\U$ is a fixed function in the space $\mathscr{C}^{1,s_0}(T,\epsilon_0)$ defined in Definition \ref{func-space} for $\epsilon_0$ sufficiently small and $s_0>d/2+2$. We look for a control function having the form $F=\mathcal{L}(\U)U_{in}$, where $\mathcal{L}(\U)$ is a bounded linear operator from $\H^0$ to $C([0,T];\H^0)$ (see \eqref{eq: defi of H^s})
, sending the initial datum $U_{in}$ to the final target $U(T)=0$. To construct the control operator we use the Hilbert Uniqueness Method, HUM for short. 
The  HUM  establishes the equivalence between the null controllability of the  system \eqref{esempio} and an observability inequality to its adjoint system, see estimate
\eqref{eq: L^2-observability inequality}. To achieve \eqref{eq: L^2-observability inequality}, we first need to diagonalize the equation \eqref{esempio}
 at the highest order (we have a full matrix of order two, due to the quasi-linear term), this is done through a change of variable $\Phi(\U)$ built in Proposition \ref{diago}. We obtain a new system $\partial_t Y=\ii\mathscr{B}(\U)Y$, with  $\mathscr{B}$ defined after \eqref{eq:diagonalized simplified adjoint equation}.
We first prove its semi-classical version, i.e. \eqref{eq: semiclassical observability inequality}, where $h>0$ is small enough. The proof of \eqref{eq: semiclassical observability inequality} is obtained by contradiction. A careful study of the propagation of semi-classical defect measures is needed, this is the point where the geometric control condition, see Definition \ref{gcc}, is used. Owing to \eqref{eq: semiclassical observability inequality}, by means of Littlewood Paley's theory, we obtain the weak observability inequality \eqref{eq: weak observability inequality}. The presence of the low frequency remainder $\|Y(0)\|_{\H^{-N}}$ in \eqref{eq: weak observability inequality} is due to the fact that we obtain \eqref{eq: semiclassical observability inequality} only for $h$ small enough. 
We use a uniqueness-compactness argument, see Section \ref{UCSO}, to remove such a remainder. In such a way, thanks to the HUM principle, we may obtain the existence of the control operator $\mathcal{L}$ as a bounded operator from $\H^0$ to $C([0,T];\H^0)$. To study the $\H^s$ regularity we need to perform some commutator estimates, in this point, it is important, due to the quasi-linear character of the equation, to use the equivalent energy norm adapted to the equation, built in Section \ref{sec:mod}. After that one has to pass to the limit in the equation \eqref{n-prob}. To do this it is important to prove several {\em contraction} estimates on the control operator, this part is performed in  Section \ref{sec: nonlinear control}. We prove that  the null controllability implies to exact controllability in Section \ref{sec: establish the exact controllability}. 

The paper is organized as follows. In Section \ref{lwp}, we recall some notions of para-differential calculus and we prove the local well-posedness of the linearized problems. Section \ref{sec: L^2 controllability for the linear system} is the core of the paper, here we prove the $L^2$ null controllability for the linearized problems. In Section \ref{sec: H^s linear control} we prove the $\H^s$ linear null controllability, while in Section \ref{sec: nonlinear control} we pass to the limit and recover the null controllability of the nonlinear equation \eqref{NLS}. Finally, in Section \ref{sec: establish the exact controllability} we show the exact controllability.

\section{Local existence of the para-linearazed equation}\label{lwp}
 \subsection{Para-differential calculus}
 In this section, we recall some results about para-differential calculus, all the proof of these results may be found in \cite{BMM1}.  \begin{definition}
Given $m,s\in\mathbb{R}$ we denote by $\Gamma^m_s$ the space of functions $a(x,\xi)$ defined on $\T^d\times \R^d$ with values in $\C$, which are $C^{\infty}$ with respect to the variable $\xi\in\R^d$ and such that for any $\beta\in \N\cup\{0\}$,
there exists a constant $C_{\beta}>0$ such that 
\begin{equation}\label{stima-simbolo}
\|\partial_{\xi}^{\beta} a(\cdot,\xi)\|_{H^s}\leq C_{\beta}\langle\xi\rangle^{m-\beta}, \quad \forall \xi\in\R.
\end{equation}
\end{definition}
We endow the space $\Gamma^m_{s}$ with the family of seminorms 
\begin{equation}\label{seminorme}
|a|_{m,s,n}:=\max_{\beta\leq n} \sup_{\xi\in\R^d}\|\langle\xi\rangle^{\beta-m}a(\cdot,\xi)\|_{H^s}.
\end{equation}
Analogously for a given Banach space $W$ we denote by $\Gamma^m_{W}$ the space of functions which verify the \eqref{stima-simbolo} with the $W$-norm instead of $H^s$, we also denote by $|a|_{m,W,n}$ the $W$ based seminorms \eqref{seminorme} with $H^s\rightsquigarrow W$.\\
We say that a symbol $a(x,\xi)$ is spectrally localised if there exists $\delta>0$ such that $\hat{a}(j,\xi)=0$ for any $|j|\geq \delta \langle\xi\rangle.$\\
Consider a function $\chi\in C^{\infty}(\R,[0,1])$ such that $\chi(\xi)=1$ if $|\xi|\leq 1.1$ and $\chi(\xi)=0$ if $|\xi|\geq 1.9$. Let $\epsilon\in(0,1)$ and define moreover $\chi_{\epsilon}(\xi):=\chi(\xi/\epsilon).$ Given $a(x,\xi)$ in $\Gamma^m_s$ we define the regularised symbol
\begin{equation*}
a_{\chi}(x,\xi):=\sum_{j\in \Z^d}\hat{a}(j,\xi)\chi_{\epsilon}(\tfrac{j}{\langle\xi\rangle})e^{\ii jx}.
\end{equation*}
 For a symbol $a(x,\x)$ in $\Gamma^m_s$
we define its Weyl and Bony-Weyl
quantization as 
\begin{equation}\label{quantiWeylcl}
\opw(a(x,\xi))g:=\frac{1}{(2\pi)}\sum_{j\in \mathbb{Z}^d}e^{\ii j x}
\sum_{k\in\mathbb{Z}^d}
\hat{a}\big(j-k,\frac{j+k}{2}\big)\widehat{g}(k),
\end{equation}
\begin{equation}\label{quantiWeyl}
\opbw(a(x,\xi))g:=\frac{1}{(2\pi)}\sum_{j\in \mathbb{Z}^d}e^{\ii j x}
\sum_{k\in\mathbb{Z}^d}
\chi_{\epsilon}\Big(\frac{|j-k|}{\langle j+k\rangle}\Big)
\hat{a}\big(j-k,\frac{j+k}{2}\big)\widehat{g}(k).
\end{equation}
We list below a series of theorems and lemmas that will be used in the paper. All the statements have been taken from \cite{BMM1}. The first one is a result concerning the action of a para-differential operator  on Sobolev spaces. This is Theorem 2.4 in \cite{BMM1}.
\begin{theorem}\label{azione}
Let $a\in\Gamma^m_{s_0}$, $s_0>d/2$ and $m\in\R$. Then $\opbw(a)$ extends as a bounded operator from ${H}^s(\T^d)$ to ${H}^{s-m}(\T^d)$ for any $s\in\R$ with estimate  
\begin{equation}\label{ac1}
\|\opbw(a)u\|_{H^{s-m}}\lesssim |a|_{m,s_0,4}\|u\|_{H^s},
\end{equation}
for any $u$ in $H^s(\T^d)$. Moreover for any $\rho\geq 0$ we have for any ${u}\in H^s(\T^d)$
\begin{equation}\label{ac2}
\|\opbw(a)u\|_{H^{s-m-\rho}}\lesssim |a|_{m,s_0-\rho,4}\|u\|_{H^s}.
\end{equation}
\end{theorem}
We now state a result regarding symbolic calculus for the composition of Bony-Weyl para-differential operators.   Given two symbols $a$ and $b$ belonging to $\Gamma^m_{s_0+\rho}$ and $\Gamma^{ m'}_{s_0+\rho}$ respectively, we define for $\rho\in(0,2]$
\begin{equation}\label{cancelletto}
a\#_{\rho} b= \begin{cases}
ab \quad \rho\in(0,1]\\
ab+\frac{1}{2\ii}\{a,b\}\quad \rho\in (1,2],\\
\end{cases}
\end{equation}
where we denoted by $\{a,b\}:=\partial_{\xi}a\partial_xb-\partial_xa\partial_{\xi}b$ the Poisson's bracket between symbols.
\begin{remark}\label{simmetrie}
According to the notation above we have $ab\in\Gamma^{m+m'}_{s_0+\rho}$ and $\{a,b\}\in\Gamma^{m+m'-1}_{s_0+\rho-1}$. Moreover $\{a,b\}=-\{b,a\}$.
\end{remark}
The following is Theorem 2.5 of \cite{BMM1}.
\begin{theorem}\label{compo}
Let $a\in\Gamma^m_{s_0+\rho}$ and $b\in\Gamma^{m'}_{s_0+\rho}$ with $m,m'\in\mathbb{R}$ and $\rho\in(0,2]$. We have $\opbw(a)\circ\opbw(b)=\opbw(a\#_{\rho}b)+R^{-\rho}(a,b)$, where the linear operator $R^{-\rho}$ is defined on $H^s$ with values in $H^{s+\rho-m-m'}$, for any $s\in\R$ and it satisfies 
\begin{equation}\label{resto}
\|R^{-\rho}(a,b)u\|_{H^{s-(m+m')+\rho}}\lesssim (|a|_{m,s_0+\rho,N}|b|_{m',s_0,N}+|a|_{m,s_0,N}|b|_{m',s_0+\rho,N})\|u\|_{H^s},
\end{equation}
where $N\geq 3d+4$.\end{theorem}


\begin{lemma}{\bf (Paraproduct).}\label{Paraproduct}
Fix $s_0>d/2$ and 
let $f,g\in H^{s}(\mathbb{T}^d;\mathbb{C})$ for $s\geq s_0$. Then
\begin{equation}\label{eq:paraproduct}
fg=\opbw(f)g+\opbw({g})f+\mathcal{R}(f,g)\,,
\end{equation}
where
\begin{equation}\label{eq:paraproduct2}
\begin{aligned}
\widehat{\mathcal{R}(f,g)}(\x)&=\frac{1}{(2\pi)^d}
\sum_{\eta\in \mathbb{Z}^d}
a(\x-\eta,\xi)\hat{f}(\x-\eta)\hat{g}(\eta)\,,\\
 |a(v,w)|&\lesssim\frac{(1+\min(|v|,|w|))^{\rho}}{(1+\max(|v|,|w|))^{\rho}}\,,
\end{aligned}
\end{equation}
for any $\rho\geq0$.
For $0\leq \rho\leq s-s_0$ one has
\begin{equation}\label{eq:paraproduct22}
\|\mathcal{R}(f,g)\|_{{H}^{s+\rho}}\lesssim\|f\|_{{H}^{s}}\|g\|_{{H}^{s}}\,.
\end{equation}
\end{lemma}

%

\subsection{Paralinearization}
We paralinearize the equation \eqref{NLS}. This is essentially Proposition 4.2 in \cite{FGI}.

\bigskip We shall use the following 
notation throughout  the rest of the  paper
\begin{equation}\label{matriciozze}
U:=\vect{u}{\bar{u}}\,, \quad F:=\vect{f}{\bar{f}},\quad  E:=\sm{1}{0}{0}{-1}\,,\quad
\uno:=\sm{1}{0}{0}{1}\,,\quad \diag(b):=b\uno\,,\,\;\; b\in\mathbb{C}\,.
\end{equation}
Besides the classical Sobolev space $H^s(\T^d,\C)$, we also consider spaces $H^s(\T^d,\C^2)$ endowed with the norm defined by
\begin{equation*}
    \|(u_1,u_2)\|^2_{H^s(\T^d,\C^2)}=\|u_1\|^2_{H^s(\T^d,\C)}+\|u_2\|^2_{H^s(\T^d,\C)}, \forall (u_1,u_2)\in H^s(\T^d,\C^2).
\end{equation*}
Moreover, we define the subspace $\H^s$ of $H^s(\T^d,\C^2)$ as 
\begin{equation}\label{eq: defi of H^s}
    \H^s=\{U=(u,\bar{u}): u\in H^s(\T^d,\C^2)\}.
\end{equation}
In particular, we define the scalar product in $\H^0$ as
\begin{equation}\label{eq: scalar product in H^0}
    \poscals{U}{V}:=\int_{\T^d}u(x)\bar{v}(x)dx+\int_{\T^d}v(x)\bar{u}(x)dx.
\end{equation}
Define the following \emph{real} symbols
\begin{equation}\label{simboa2}
\begin{aligned}
a_2(x):=&\, \left[g_1'(|u|^2)\right]^2|u|^2\,,
\quad b_2(x):=
\left[g_1'(|u|^2)\right]^2u^2,\\
\vec{a}_1(x)\cdot\xi:=&\,\left[g_1'(|u|^2)\right]^2
\sum_{j=1}^d\Im(u\bar{u}_{x_j})\xi_j\,, \quad \xi=(\xi_1,\ldots,\xi_d)\,. 
\end{aligned}
\end{equation}
We define also the matrix of functions
\begin{equation}\label{matriceA2}
A_2(x):=A_2(U;x):=\sm{1+a_2(U;x)}{b_2(U;x)}{\ov{b_2(U;x)}}{1+a_2(U;x)}=
\sm{1+a_2(x)}{b_2(x)}{\ov{b_2(x)}}{1+a_2(x)}
\end{equation}
with $a_2(x)$ and $b_2(x)$ defined in \eqref{simboa2} and 
\begin{equation}\label{eq: operator A}
\A(U):=-\big(E \opbw(A_2(U;x)|\xi|^2)+ \diag(\vec{a}_1(U;x)\cdot\xi)\big)
\end{equation}

In the following proposition we rewrite \eqref{NLS} as a system of para-differential equations. This has been already proven in \cite{FIJMPA}, one can also look at \cite{BMM1}.
\begin{proposition}[Paralinearization]\label{paralinearization}
The equation \eqref{NLS} may be rewritten as  the following system of para-differential equations.
\begin{equation}\label{NLSpara}
\partial_tU=\ii\A(U)U+ R(U)U-\ii E F,
\end{equation}
where the matrices of symbols are defined in \eqref{matriceA2}, \eqref{simboa2}, and $R$ is a semilinear remainder satisfying the following. For any $\sigma\geq s_0>d/2$ and any $U, V$ in $\H^{\sigma+2}$ we have
\begin{align}
&\|R(U)\|_{\H^{\sigma}}\lesssim \|U\|_{\H^{\s}}, \label{R1}\\
&\|R(U)W-R(V)W\|_{\H^{\s}}\lesssim \|U-V\|_{\H^{s_0}}\|W\|_{\H^{\s}}.\label{R3}
\end{align}
Moreover if $s-2\geq 0$ we have
\begin{equation}\label{Lip-simbolo}
\|(\A(U)-\A(V))W\|_{\H^{s-2}}\lesssim \|U-V\|_{\H^{s_0}}\|W\|_{\H^s}.
\end{equation}
\end{proposition}

\subsection{Diagonalization of the linear system}
We consider the linear system associated to \eqref{NLSpara} with $F=0$ and we diagonalize it at the highest order. Define the space of functions as follows.
\begin{definition}\label{func-space}
For $s\in\mathbb{R}$, $\epsilon_0>0$, $T>0$, we say that $\U\in\mathscr{C}^{1,s}(T,\epsilon_0)$ if $U\in C([0,T];\H^s)\cap Lip([0,T]; {\H}^{s-2})$ and the norm
\begin{equation*}
    ||\U||_{\mathscr{C}^{1,s}}:=||\U||_{C([0,T];\H^s)}+||\partial_t \U||_{L^{\infty}([0,T]; {\H}^{s-2})}\leq\epsilon_0.
\end{equation*}
Similarly, we say that $F\in \mathscr{C}^{0,s}(T,\epsilon_0)$ if $F\in C([0,T];\H^s)$ and the norm
\begin{equation*}
    ||F||_{\mathscr{C}^{0,s}}:=||F||_{C([0,T];\H^s)}\leq\epsilon_0.
\end{equation*} 
\end{definition}
Consider the equation
\begin{equation}\label{NLSparalin}
\partial_tU=-\ii E \opbw(A_2(\underline{U};x)|\xi|^2)U- \ii\diag(\vec{a}_1(\underline{U};x)\cdot\xi)U.
\end{equation}
The eigenvalues of the matrix $E A_2(\underline{U};x)$ are equal to 
\begin{equation}\label{autovalori}
\lambda_{\pm}(\underline{U};x):=\pm\lambda(\underline{U};x):=\pm\sqrt{1+2|\underline{u}|^2[g_1'(|\underline{u}|^2)]^2}\,.
\end{equation}
We denote by $S$  matrix of the eigenvectors of $E(\uno+{A}_2(x))$, more explicitly
\begin{equation}\label{matriceS}
\begin{aligned}
&S=\left(\begin{matrix}
s_1 &s_2\\
\bar{s}_2 & s_1
\end{matrix}\right)\,, 
\qquad S^{-1}=\left(\begin{matrix}
s_1 &-s_2\\
-\bar{s}_2 & s_1
\end{matrix}\right)\,, \\
&s_1(\underline{U};x):=\frac{1+|\underline{u}|^2[g_1'(|\underline{u}|^2)]^2+\lambda(x)}
{\sqrt{2\lambda(x)(1+[g_1'(|\underline{u}|^2)]^2|\underline{u}|^2+\lambda(x))}}\,, \\
&s_2(\underline{U};x):=\frac{-\underline{u}^2[g_1'(|\underline{u}|^2)]^2}{\sqrt{2\lambda(x)(1+[g_1'(|\underline{u}|^2)]^2|\underline{u}|^2+\lambda(x))}}\,.
\end{aligned}
\end{equation}
With this setting we have
\begin{equation}\label{diago_algebra}
S^{-1}E(\uno+{A}_2(\underline{U};x))S=E\diag(\lambda (\underline{U};x))\,,\quad s_1^2-|s_2|^2=1\,,
\end{equation}
where we have used the notation \eqref{matriciozze}. We have the following.
\begin{proposition}[Diagonalization]\label{diago}
Let $s_0>d/2+2$ and $\underline{U}$ be in $\mathscr{C}^{1,s_0}(T,\epsilon_0)$ (see Definition \ref{func-space}) with $\epsilon_0$ small enough. For any $t\in [0,T)$, there exists a linear and  invertible map $\Phi(\underline{U})$ from $\H^s$ to $\H^s$ for any $s\geq s_0$ such that if $U$ is a solution of \eqref{NLSparalin}, then $W=\Phi(\underline{U})U$ solves the following
\begin{equation}\label{paradiag}
\partial_tW=-\ii E\opbw(\diag(\lambda(\underline{U};x))|\xi|^2)W-\ii \opbw(\diag(\vec{a}_1(\underline{U};x)\cdot\xi) W+\tilde{R}(\underline{U})W.
\end{equation}
The operator linear operator $\tilde{R}(\underline{U})$ is bounded from $\H^s$ to $\H^s$ for any $s$. Moreover, we have the estimates
\begin{align}
&\|(\Phi(V_1)-\uno)W\|_{\H^s}\lesssim \|V_1\|_{\H^{s_0}}\|W\|_{\H^s}\label{phi1}, \\
&\|(\Phi(V_1)-\Phi(V_2))W\|_{\H^s}\lesssim \|V_1-V_2\|_{\H^{s_0}}\|W\|_{\H^s}\label{phi_2},\\
&\|\tilde{R}(V_1)W\|_{\H^s}\lesssim\|V_1\|_{\H^{s_0}}\|W\|_{\H^s},\label{est_R}
\end{align}
for any $s\in\R$ and $V_1, V_2$ in $\H^{s_0}$ and $W$ in $\H^s$.
\end{proposition}
\begin{proof}
We define $\Phi(\underline{U};x):=\opbw(S^{-1}(\underline{U};x))$. By using Theorem \ref{compo} with $\rho=2$, $m, m'=0$ and $Q=R^{-\rho}$ we obtain $ \opbw(S(\underline{U};x))\circ\Phi=\uno+Q(\underline{U})$, where $\|Q(\underline{U})W\|_{\H^s}\lesssim\|\underline{U}\|_{\H^{s_0}}\|W\|_{\H^s}$. Since $\underline{U}$ is in $\mathscr{C}^{1,s}(T,\epsilon_0)$ with $\epsilon_0$ small we may invert the operator $\uno+Q(\underline{U})$ by means of Neumann series, therefore we obtain $\Phi^{-1}(\underline{U})=(\uno+Q(\underline{U}))^{-1}\circ\opbw(S(\underline{U}))$.\\
The estimate \eqref{phi1} is a direct computation, indeed the off diagonal term of the matrix $S^{-1}$ already generates terms that satisfy such an estimate (by means of Theorem \ref{azione}), while we observe that the term on the diagonal $s_1-1=\frac{1+a-\lambda}{\sqrt{2\lambda}(\sqrt{1+a+\lambda}+\sqrt{2\lambda})}=-a+O(a^2)$, where we set $a:=|\underline{u}|^2[g_1'(|\underline{u}|^2)]^2$, therefore we conclude by using again Theorem \ref{azione}. The estimate \eqref{phi_2} may be obtained similarly, by using  \eqref{F2} and Theorems \ref{azione}, \ref{compo}. We prove \eqref{paradiag}, deriving with respect to $t$ we have
\begin{align*}
\partial_tW&=\opbw(\partial_t(S^{-1}(\underline{U})))U+\Phi(\underline{U})\partial_tU\\
&=\opbw(\partial_t(S^{-1}))\Phi^{-1}W-\ii\Phi E\opbw(A_2|\xi|^2)\Phi^{-1}W\\
&\quad +\ii\Phi\opbw(\vec{a}_1\cdot\xi)\Phi^{-1}W.
\end{align*}
We study the main term  $\Phi E\opbw(A_2)\Phi^{-1}W$. By Neumann series, since $Q(\underline{U})$ is a linear operator gaining two derivatives, we have that $\Phi^{-1}(\underline{U})=\uno+ \frak{R}(\underline{U})$, with $\frak{R}(\underline{U})$ gaining two derivatives and having estimates as in \eqref{est_R}. We use this fact,  Theorem \ref{compo} and the fact that $S$ is the matrix of the eigenvectors of $EA_2$ to conclude that $\Phi E\opbw(A_2(\underline{U})|\xi|^2)\Phi^{-1}W=E\opbw(\diag(\lambda(\underline{U}))|\xi|^2)W$ modulo contributions to $\tilde{R}(\underline{U})W$. The conjugation of the operator of order one may be treated similarly and the term $\opbw(\partial_t(S^{-1}))\Phi^{-1}W$ may be absorbed in the remainder $\tilde{R}(\underline{U})$, because the function $\underline{U}$ is assumed to be in $\mathscr{C}^{1,s}(T,\epsilon_0)$. We eventually obtained the \eqref{paradiag}.
\end{proof}

\subsection{Solutions of the linear problems}\label{sec:mod}
In order to show the existence of the solutions to the problem \eqref{paradiag} we shall prove energy inequality and use a standard regularization scheme. The following is inspired to  \cite{BMM1, iandoliKdV}. We need to introduce the following equivalent norm on $\H^{\sigma}$, $\sigma\geq 0$, which is adapted to the equation \eqref{paradiag}
\begin{equation}\label{modified}
\|W\|_{\sigma,\underline{U}}^2=\langle \Lambda^{2\sigma}(\underline{U})W,W\rangle_{L^2},
\end{equation}
where we have defined
\begin{equation}\label{big-Lambda}
 \Lambda^{2\s}(\underline{U}):=\opbw\left(\diag\left(1+\lambda(\underline{U};x)|\xi|^{2})^{\s}\right)\right),
\end{equation}
and where $\underline{U}$ is assumed to be in $\mathscr{C}^{1,s_0}(T,\epsilon_0)$ (see Definition \ref{func-space}). One can show that the norm $\|\cdot\|_{\sigma,\underline{U}}$ is equivalent to the $\H^{\s}$ norm, since we work with small data. This follows, for instance, as in the proof of Lemma \ref{lem: regularity of the weight op} with $h$ therein equal to one.

\begin{proposition}\label{linear existence}
Let $\underline{U}$ be in $\mathscr{C}^{1,s_0}(T,\epsilon_0)$ with $\epsilon_0$ sufficiently small, $\sigma\geq 0$ and $s_0>d/2+2$. 
 Then the Cauchy problem \eqref{paradiag}, with initial condition $W(0,x)=\Phi(\underline{U}(0))U_0$ admits a unique solution $W\in C([0,T);\H^\s)$ which satisfies the estimate
$\|W\|_{C^0([0,T),\H^{\s})}\lesssim \|U_0\|_{\H^{\s}}$.
\end{proposition}
\begin{proof}
We consider $U_0\in C^{\infty}$, the $\H^{\s}$ case follows by standard approximation arguments. For any  $\varepsilon>0$ and, we introduce the cut-off function $\chi_{\varepsilon}:=\chi(\varepsilon \lambda(\underline{U};x)|\xi|^2)$, with $\chi$ smooth function supported in a neighborhood of the origin. We consider the smoothed homogeneous system
\begin{equation}\label{paradiag-smooth}
\begin{aligned}
\partial_tW_{\varepsilon}=&-\ii E\opbw(\chi_\varepsilon\diag(\lambda(\underline{U};x))|\xi|^2)W_{\e}-\ii \opbw(\chi_{\varepsilon}\diag(\vec{a}_1(\underline{U};x))\cdot\xi) W_{\e}+\tilde{R}(\underline{U})W_{\e}.\\
\end{aligned}
\end{equation}
At fixed $\varepsilon>0$, \eqref{paradiag-smooth} is a Banach space ODE, therefore it admits a unique solution $W_{\varepsilon}$ which belongs to the space $C^1([0,T_{\varepsilon}),\H^{\s})$, for some $T_{\varepsilon}>0$. In order to show that $W_{\varepsilon}$ converges to a solution of \eqref{paradiag} which is in $C^1([0,T),\H^{\s})$ for $T>0$ independent of $\varepsilon$ we perform an energy estimate with constants independent of $\varepsilon$. We have 
\begin{align*}
\tfrac{d}{dt}\|W_{\e}(t,\cdot)\|_{\s,\underline{U}}^2=&\langle\ii [E\opbw(\chi_\varepsilon\diag(\lambda|\xi|^2)),\Lambda^{\s}]W_{\e},W_{\e}\rangle_{L^2}\\
							  &+\langle \ii [\Lambda^{\s},\opbw(\chi_{\e}\diag(\vec{a}_1\cdot\xi))]W_{\e},W_{\e}\rangle_{L^2}\\
							  &+2\Re\langle\Lambda^{\s} \tilde{R}(\underline{U}) W_{\e},W_{\e}\rangle_{L^2}\\
							  &+\langle \opbw(\tfrac{d}{dt}(1+\lambda|\xi|^{2})^{\s})W_{\e},W_{\e}\rangle.
							  \end{align*}

The main term is the one in the r.h.s. of the first line, we use Theorem \ref{compo} with $\rho=2$ and we see that the first two contributions are equal to zero, in particular, the Poisson bracket between the symbols is equal to zero because of the choice of the cut-off function as a function of $\lambda |\xi|^2$. In such a way, we may bound the first term in the r.h.s. of the first line by $C\|W_{\e}\|_{\underline{U},\s}^2$, for some positive constant $C$. The term in the second line is bounded by the same quantity because  of Theorem \ref{compo} applied with $\rho=1$. The term in the third line is bounded by the same quantity by means of Cauchy-Schwartz inequality and the estimate \eqref{est_R}. 
The term in the fourth line may be bounded by the same quantity by using the fact that $\underline{U}$ is in $\mathscr{C}^{1,s_0}(T,\epsilon_0)$.
Therefore, after using the equivalence $\|\cdot\|_{\U,\sigma}\sim\|\cdot\|_{\H^{\s}}$, the solution $W_{\varepsilon}(t,x)$ of \eqref{paradiag-smooth} exists and it has estimates independent of $\varepsilon$ by Gronwall  inequality $\|W_{\varepsilon}(t,x)\|_{\H^{\s}}\lesssim\|U_0\|_{\H^{\s}}$. Therefore the limit for $\e\rightarrow 0$ converges to a solution of the same equation without the cut-off function $\chi_{\e}$ by the Ascoli-Arzelà Theorem. 
\end{proof}

\begin{corollary}\label{estimates-near-id}
The function $U=\Phi(\underline{U})^{-1} W$ solves \eqref{NLSpara} with $F=0$ and satisfies the estimate $$\|U\|_{C^0([0,T],\H^s)}\leq (1+\mathcal{F}(\|\underline{U}\|_{\H^{s_0}}))(\|U_0\|_{\mathcal{H}^{s}}),$$ where $\F(\cdot)$ is a non decreasing function such that $\F(0)=0$. When $F$ is different from zero one has $\|U\|_{C^0([0,T],\H^s)}\leq (1+\mathcal{F}(\|\underline{U}\|_{\H^{s_0}}))(\|U_0\|_{\mathcal{H}^{s}}+\|F\|_{L^1([0,T],\H^s)})$.
\end{corollary}
\begin{proof}
The first sentence is a consequence of Proposition \ref{linear existence}, when $F\neq 0$ one has to use Duhamel formulation of \eqref{NLSpara}.
\end{proof}
\section{$L^2$ null controllability for the linear system}\label{sec: L^2 controllability for the linear system}
In this section, we mainly consider the $L^2$ controllability for the linearized system. In Proposition \ref{paralinearization}, we proved that \eqref{NLS} is equivalent to
\begin{equation}\label{eq: nonlinear control problem}
\partial_tU=-\ii E \opbw(A_2(U;x)|\xi|^2)U- \ii\opbw(\diag(\vec{a}_1(U;x)\cdot\xi))U+ R(U)-\ii \chi_T\varphi_{\omega}E F.
\end{equation}

We introduce an iterative scheme that reduces the study of the nonlinear control problem to the linear one. For certain $s_0>d/2+2$  and $\epsilon_0>0$, we fix $\underline{U}\in\mathscr{C}^{1,s_0}(T,\epsilon_0)$ (see Definition \ref{func-space}) and we consider the null controllability of the following linear equation:
\begin{equation}\label{eq:linearized system with R}
    \partial_t U=\ii \mathscr{A}(\underline{U})U+
R(\underline{U})U-\ii \chi_T\varphi_{\omega}E F\,,
\end{equation}
where $\mathscr{A}(\underline{U})=-E \opbw(A_2(\underline{U};x)|\xi|^2)- \opbw(\diag(\vec{a}_1(\underline{U};x)\cdot\xi)$.
We show that for $\epsilon_0$ sufficiently small, there exists a linear operator 
\begin{equation}
\mathscr{L}(\underline{U}):\H^s\rightarrow C([0,T],\H^s),    
\end{equation}
which null-controls \eqref{eq:linearized system with R}, that is, for any $U_{in}\in\H^s$,
\begin{equation*}
    F=\mathscr{L}(\underline{U})U_{in}\in C([0,T],\H^s)
\end{equation*}
sends the initial datum $U(0)=U_{in}$ to the final target $U(T)=0$. Therefore, for $U_{in}\in \H^s$ such that $||U_{in}||_{\H^s}<\epsilon_0$, the iterative scheme proceeds as follows. Set $U^0=0$, $F^0=0$, and for $n\geq0$, set $(U^{n+1},F^{n+1})\in C([0,T],\H^s)\times C([0,T],\H^s)$ by letting $F^{n+1}=\mathcal{L}(U^n)U_{in}$ and $U^{n+1}$ be the solution to the system 
\begin{equation}\label{n-prob}
    \partial_t U^{n+1}=\ii E\mathscr{A}(U^n)U^{n+1}+
R(U^n)U^{n+1}-\ii \chi_T\varphi_{\omega}E F^{n+1},\quad U^{n+1}(0)=U_{in}, U^{n+1}(T)=0.
\end{equation}
We focus on the construction of the operator $\mathcal{L}$. 
\subsection{Hilbert Uniqueness Method}
The term $R(\underline{U})$ is a smoothing perturbation. We focus on the simplified para-differential system:
\begin{equation}\label{eq: simplified control system-L^2}
   \partial_t U=\ii \mathscr{A}(\underline{U})U-\ii \chi_T\varphi_{\omega}E F\,, 
\end{equation}
where $\mathscr{A}(\underline{U})=-E \opbw(A_2(\underline{U};x)|\xi|^2)-\opbw(\diag(\vec{a}_1(\underline{U};x)\cdot\xi)$.

The operator $\ii\mathscr{A}(\underline{U})$ is skew-self-adjoint with respect to the scalar product on the Hilbert space $\H^0$, so that 
 the adjoint system is
\begin{equation}\label{eq: adjoint system}
    \partial_t V=-\ii \mathscr{A}(\underline{U})V.
\end{equation}
Now HUM proceeds as follows. Define the range operator $\mathcal{R}=\mathcal{R}(\underline{U})$ by
\begin{equation*}
    \mathcal{R}: L^2([0,T],\H^0)\rightarrow \H^0, F\mapsto U_{in}{:=U(0)},
\end{equation*}
where $U\in C([0,T],\H^0)$ is the unique solution to the Cauchy problem
\begin{equation*}
    \partial_t U=\ii \mathscr{A}(\underline{U})U+F,\quad U(T)=0.
\end{equation*}
\begin{remark}
Here we  suppose $T>0$ is small enough to ensure the existence of the solutions. 
\end{remark}
\noindent We define the solution operator $\mathcal{S}=\mathcal{S}(\underline{U})$ by 
\begin{equation*}
    \mathcal{S}: \H^0\rightarrow C([0,T],\H^0)\subset L^2([0,T],\H^0), V(0)\mapsto V,
\end{equation*}
where $V\in C([0,T],\H^0)$ is the unique solution to the adjoint system \eqref{eq: adjoint system} with initial datum $V(0)$. Then we have the following duality proposition.
\begin{proposition}
Let $F\in L^2([0,T],\H^0)$, and $V(0)\in \H^0$. Then the duality holds as follows:
\begin{equation}\label{duality}
    -(\mathcal{R}\chi_T\varphi_{\omega}F,V(0))_{\H^0}=(\chi_T\varphi_{\omega}F,\mathcal{S}V(0))_{L^2([0,T],\H^0)}.
\end{equation}
Formally, we could say that $\mathcal{R}\circ(\chi_T\varphi_{\omega})=-((\chi_T\varphi_{\omega})\circ\mathcal{S})^*$. \begin{proof}
We consider $F\in L^2([0,T],\H^{\infty})$ and $V(0)\in \H^{\infty}$, the $\H^0$ case follows by standard approximation arguments. Since $V=\mathcal{S}(\underline{U})V(0)$ is the solution to 
\begin{equation*}
    \partial_t V=-\ii \mathscr{A}(\underline{U})V, V|_{t=0}=V(0)\in \H^{\infty},
\end{equation*}
we know that $V\in C^1([0,T],\H^{\infty})$. Using the equation \eqref{eq: simplified control system-L^2}, we obtain that
\begin{align*}
    (\chi_T\varphi_{\omega}E F,\mathcal{S}V(0))_{L^2([0,T],\H^0)}&=(\ii \partial_t U+\mathscr{A}(\underline{U})U,\mathcal{S}V(0))_{L^2([0,T],\H^0)}\\
    &=(\ii \partial_t U,\mathcal{S}V(0))_{L^2([0,T],\H^0)}+(\mathscr{A}(\underline{U})U,\mathcal{S}V(0))_{L^2([0,T],\H^0)}.
\end{align*}
We already know that $\mathscr{A}(\underline{U})$ is self-adjoint, therefore we derive that
\begin{equation*}
    (\chi_T\varphi_{\omega}E F,\mathcal{S}V(0))_{L^2([0,T],\H^0)}
    =(\ii \partial_t U,\mathcal{S}V(0))_{L^2([0,T],\H^0)}+(U,\mathscr{A}(\underline{U})\mathcal{S}V(0))_{L^2([0,T],\H^0)}.
\end{equation*}
According to the Newton-Leibniz formula, 
\begin{align*}
    (\chi_T\varphi_{\omega}E F,&\mathcal{S}V(0))_{L^2([0,T],\H^0)}\\
   & =(\ii U,\mathcal{S}V(0))_{\H^0}|_0^T-(\ii  U,\partial_t\mathcal{S}V(0))_{L^2([0,T],\H^0)}+(U,\mathscr{A}(\underline{U})\mathcal{S}V(0))_{L^2([0,T],\H^0)}.
\end{align*}
Using the equation \eqref{eq: adjoint system}, we obtain
$
     (\chi_T\varphi_{\omega}E F,\mathcal{S}V(0))_{L^2([0,T],\H^0)}=(\ii U,\mathcal{S}V(0))_{\H^0}|_0^T.
$
Since $U(T)=0$, we deduce that
$
     (\chi_T\varphi_{\omega}E F,\mathcal{S}V(0))_{L^2([0,T],\H^0)}=-(\ii U(0),V(0))_{\H^0}.
$
Since $\mathcal{R}(-\ii\varphi_{\omega}EF)=U(0)$, we conclude with 
$ (\chi_T\varphi_{\omega}E F,\mathcal{S}V(0))_{L^2([0,T],\H^0)}=-(\mathcal{R}\chi_T\varphi_{\omega}EF ,V(0))_{\H^0}.$
\end{proof}
\end{proposition}
We introduce the HUM operator denoted by $\mathscr{K}=\mathscr{K}(\underline{U})$ given by 
\begin{equation}\label{defi: HUM op}
    \mathscr{K}(\underline{U})=-\mathcal{R}\chi_T^2\varphi_{\omega}^2\mathcal{S}:\H^0\rightarrow \H^0.
\end{equation}
\begin{remark}
According to the energy estimates in Corollary \ref{estimates-near-id}, one has $\mathscr{K}: \H^s\rightarrow \H^s$ for any $s\geq0$. 
\end{remark}

We adopt the following notation for the rest of the paper.
We denote by $\F(\|\underline{U}\|_{\H^{s_0}})$ a positive constant depending on $s_0$ and on the norm of $\underline{U}$ in the space $\H^{s_0}$, this constant may change from line to line.

\begin{proposition}\label{prop:L^2 observability}
Suppose that $\omega$ satisfies the geometric control condition, $s_0>d/2+2$, $T>0$, $\epsilon_0>0$, and $\underline{U}\in \mathscr{C}^{1,s_0}(T,\epsilon_0)$. Then for $\epsilon_0$ sufficiently small, the $L^2$-observability of the system \eqref{eq: adjoint system} holds. That is, for all solutions $V$, with initial datum  $V(0)$ in $\H^0$, we have the following inequality:
\begin{equation}\label{eq: L^2-observability inequality}
    ||V(0)||^2_{\H^0}\leq \F(\|\underline{U}\|_{\H^{s_0}})\int_0^T||\chi_T\varphi_{\omega}V(t)||^2_{\H^0}dt.
\end{equation}
\end{proposition}
The proof of Proposition \ref{prop:L^2 observability} is postponed after the statement of Proposition  \ref{prop:L^2 observability-diagonal}.
As a consequence of Proposition \ref{prop:L^2 observability}, we can prove the invertibility of the HUM operator \eqref{defi: HUM op}.

\begin{proposition}\label{prop: HUM L^2-iosmorphism}
Suppose that $\omega$ satisfies the geometric control condition, $s_0>d/2+2$, $T>0$, $\epsilon_0>0$, and $\underline{U}\in \mathscr{C}^{1,s_0}(T,\epsilon_0)$. Then for $\epsilon_0$ sufficiently small, the HUM operator $\mathscr{K}$ defines an isomorphism on $\H^0$ with
\begin{equation*}
    ||\mathscr{K}||_{\mathcal{L}(\H^0,\H^0)}+||\mathscr{K}^{-1}||_{\mathcal{L}(\H^0,\H^0)}\leq \F(\|\underline{U}\|_{\H^{s_0}}).
\end{equation*}
\end{proposition}
\begin{proof}
We can reformulate the \eqref{eq: L^2-observability inequality} in terms of the solution operator $\mathcal{S}$ as follows:
\begin{equation*}
    ||V(0)||^2_{\H^0}\leq \F(\|\underline{U}\|_{\H^{s_0}}) ||\chi_T\varphi_{\omega}\mathcal{S}V(0)||^2_{L^2([0,T],\H^0)},
\end{equation*}
which could be seen as the coercivity of the operator $\chi_T\varphi_{\omega}\mathcal{S}$. Therefore, we consider the continuous form on $\H^0$ given by
\begin{equation}
    \alpha(U,V)=(\mathscr{K}U,V)_{\H^0}.
\end{equation}
Fix $V$ in $\H^0$, by using \eqref{duality} and \eqref{eq: L^2-observability inequality} we obtain
\begin{equation*}
\begin{aligned}
    \alpha(V,V)&=(\mathscr{K}V,V)_{\H^0}
               =(-\mathcal{R}\chi_T^2\varphi_{\omega}^2\mathcal{S}V,V)_{\H^0}\\
               &=(\chi_T\varphi_{\omega}\mathcal{S}V,\chi_T\varphi_{\omega}\mathcal{S}V)_{L^2([0,T],\H^0)}\\
               &=||\chi_T\varphi_{\omega}\mathcal{S}V||^2_{L^2([0,T],\H^0)}\geq \F(\|\underline{U}\|_{\H^{s_0}})^{-1} ||V||^2_{\H^0},
\end{aligned}
\end{equation*}
which proves that  $\alpha$ is  coercive. By Lax-Milgram's theorem, for $U_{in}\in \H^0$, there exists a unique $U_{\mathcal{R}}\in\H^0$ such that 
\begin{equation*}
    \alpha(U_{\mathcal{R}},V)=(U_{in},V)_{\H^0},\forall V\in\H^0.
\end{equation*}
We eventually proved the invertibility of the HUM operator $\mathscr{K}$ in \eqref{defi: HUM op},
\begin{equation*}
    \mathscr{K}U_{\mathcal{R}}=U_{in}, \quad||U_{\mathcal{R}}||_{\H^0}\leq \F(\|\underline{U}\|_{\H^{s_0}}) ||U_{in}||_{\H^0}
\end{equation*}
\end{proof}
Using the HUM operator, we construct our control operator $\mathscr{L}$. We set 
\begin{equation*}
F=-\ii\chi_T\varphi_{\omega}E\mathcal{S}U_{\mathcal{R}}=-\ii\chi_T\varphi_{\omega}E\mathcal{S}\mathscr{K}^{-1}U_{in}\in C([0,T],\H^0).   
\end{equation*}
By definition, we have 
\begin{align*}
    \mathcal{R}(-\ii\chi_T\varphi_{\omega}EF)&=-\mathcal{R}\chi_T^2\varphi^2_{\omega}\mathcal{S}\mathscr{K}^{-1}U_{in}
    =\mathscr{K}\mathscr{K}^{-1}U_{in}
    =U_{in}.
\end{align*}
This implies that $F=-\ii\chi_T\varphi_{\omega}E\mathcal{S}\mathscr{K}^{-1}U_{in}$ null-controls \eqref{eq: simplified control system-L^2}. Therefore, on $\H^0$, $\mathscr{L}=-\ii\chi_T\varphi_{\omega}E\mathcal{S}\mathscr{K}^{-1}:\H^0\rightarrow C([0,T],\H^0)$. In virtue of this reasoning, the main effort is to prove the $L^2$-observability i.e., Proposition \ref{prop:L^2 observability}.

\subsection{Observability  of the diagonalized adjoint system}
We recall that in Proposition \ref{diago}, we showed the existence of a transformation  $\Phi$ such that $W=\Phi(\underline{U})V$ satisfies the following equations:
\begin{equation}\label{eq: diagnolized equation with perturbation}
    \partial_tW=\ii E\opbw(\diag(\lambda(\underline{U};x))|\xi|^2)W+\ii \opbw(\diag(\vec{a}_1(\underline{U};x))\cdot\xi) W+\tilde{R}(\underline{U})W,
\end{equation}
where $\lambda(\underline{U};x)$ and $ \vec{a}_1(\underline{U};x)$ are given in Proposition \ref{diago}
and
the remainder $\tilde{R}(\underline{U})$ satisfies the estimate
\begin{equation}
\|\tilde{R}(\underline{U})W\|_{\H^0}\lesssim  \epsilon_0\|W\|_{\H^0}\lesssim\epsilon_0\|W(0)\|_{\H^0}\,.
\end{equation}
Now consider the \emph{leading order} equation
\begin{equation}\label{eq:diagonalized simplified adjoint equation}
    \partial_t Y=\ii\mathscr{B}(\underline{U})Y,
\end{equation}
where $\mathscr{B}(\underline{U}):=E\opbw(\diag(\lambda(\underline{U};x))|\xi|^2)+\opbw(\diag(\vec{a}_1(\underline{U};x))\cdot\xi)$. We shall prove the $L^2$-observability for the solutions $Y$.
\begin{proposition}\label{prop:L^2 observability-diagonal}
Suppose that $\omega$ satisfies the geometric control condition, $s_0>d/2+2$, $T>0$, $\epsilon_0>0$, and $\underline{U}\in \mathscr{C}^{1,s_0}(T,\epsilon_0)$. Then for $\epsilon_0$ sufficiently small, the $\H^0$-observability of the system \eqref{eq:diagonalized simplified adjoint equation} holds. That is, for all solutions $Y$, with initial data in $\H^0$, we have the following inequality:
\begin{equation}\label{eq: L^2-observability inequality-diagonal}
    ||Y(0)||^2_{\H^0}\leq \F(\|\underline{U}\|_{\H^{s_0}})\int_0^T||\chi_T\varphi_{\omega}Y(t)||^2_{\H^0}dt.
\end{equation}
\end{proposition}
The proof of the above proposition is postponed after the statement of Proposition \ref{prop:L^2 observability-pseudo}. Owing to Proposition \ref{prop:L^2 observability-diagonal} we can prove  Proposition \ref{prop:L^2 observability}. 
\begin{proof}[Proof of Proposition \ref{prop:L^2 observability}]
For every solution $W\in C([0,T],\H^0)$ of the equation \eqref{eq: diagnolized equation with perturbation}, we write $W=Y+Z$, with $Z$ satisfying the equation
\begin{equation}\label{eq: diagnolized equation with zero initial}
    \partial_t Z=\ii\mathscr{B}(\underline{U}) Z+\tilde{\mathcal{R}}(\underline{U})W,\quad Z|_{t=0}=0. 
\end{equation}
And moreover, we suppose that $Y,Z\in C([0,T],\H^0)$. According to the well-posedness of the \eqref{eq: diagnolized equation with zero initial}, we have 
\begin{align*}
    ||Z||_{C([0,T],\H^0)}\lesssim ||\tilde{\mathcal{R}}(\underline{U})W||_{L^1([0,T],\H^0)}\lesssim \epsilon_0\|W(0)\|_{\H^0}.
\end{align*}
Applying Proposition \ref{prop:L^2 observability-diagonal}, we know that
\begin{align*}
||W(0)||^2_{\H^0}=||Y(0)||^2_{\H^0}&\leq \F(\|\underline{U}\|_{\H^{s_0}})\int_0^T||\chi_T\varphi_{\omega}Y(t)||^2_{\H^0}dt\\
           &\leq \F(\|\underline{U}\|_{\H^{s_0}}) \big(\int_0^T||\chi_T\varphi_{\omega}W(t)||^2_{\H^0}dt+\int_0^T||\chi_T\varphi_{\omega}Z(t)||^2_{\H^0}dt\big)\\
           &\leq \F(\|\underline{U}\|_{\H^{s_0}}) \big(\int_0^T||\chi_T\varphi_{\omega}W(t)||^2_{\H^0}dt+T||Z(t)||^2_{C([0,T],\H^0)}\big)\\
           &\leq \F(\|\underline{U}\|_{\H^{s_0}}) \big(\int_0^T||\chi_T\varphi_{\omega}W(t)||^2_{\H^0}dt+T\epsilon_0^2\|W(0)\|^2_{\H^0}\big).
\end{align*}
For $\epsilon_0$ sufficiently small, we obtain $    ||W(0)||^2_{\H^0}\leq \F(\|\underline{U}\|_{\H^{s_0}}) \int_0^T||\chi_T\varphi_{\omega}W(t)||^2_{\H^0}dt.$

\end{proof}
As a consequence, we only need to prove Proposition \ref{prop:L^2 observability-diagonal}.

\subsection{Reduction to the pseudo-differential equations}
In this section, we consider the $\H^0$-observability of the  para-differential equation  \eqref{eq:diagonalized simplified adjoint equation}.
We rewrite the para-differential operator
\begin{equation}
    \mathscr{B}(\underline{U})=\mathscr{B}^w(\underline{U})+R^w_{\mathscr{B}},
\end{equation}
where the pseudo-differential operator $\mathscr{B}^w=\mathscr{B}^w(\underline{U})$ is defined by 
\be
\mathscr{B}^w(\underline{U})=E\opw(\diag(\lambda(\underline{U};x))|\xi|^2)+\opw(\diag(\vec{a}_1(\underline{U};x)\cdot\xi)
,
\ee
 for the remainder term $R^w_{\mathscr{B}}$, we have the following proposition.
\begin{proposition}
Suppose that $s_0>d/2+2$, $T>0$, $\epsilon_0>0$, and $\underline{U}\in \mathscr{C}^{1,s_0}(T,\epsilon_0)$. Then for $\epsilon_0$ sufficiently small, we have
\begin{equation*}
    R^w_{\mathscr{B}}\in L^{\infty}([0,T],\mathcal{L}(\H^0,\H^0))\cap C([0,T],\mathcal{L}(\H^0,\H^{-2})),
\end{equation*}
with the estimates 
\begin{equation*}
    ||R^w_{\mathscr{B}}||_{L^{\infty}([0,T],\mathcal{L}(\H^0,\H^0))}\lesssim \epsilon_0.
\end{equation*}
\end{proposition}
\begin{proof}
This is a consequence of Bony paralinearization formula. One can find more details in the book \cite{Metivier}.
\end{proof}
Now we focus on the pseudo-differential equation
\be\label{eq:pseudo-equation without perturbation}
D_t V=\mathscr{B}^w V, \quad D_t:=\frac{\partial_t}{\ii}.
\ee
\begin{proposition}\label{prop:L^2 observability-pseudo}
Suppose that $\omega$ satisfies the geometric control condition, $s_0>d/2+2$, $T>0$, $\epsilon_0>0$, and $\underline{U}\in \mathscr{C}^{1,s_0}(T,\epsilon_0)$. Then for $\epsilon_0$ sufficiently small, the $\H^0$-observability of the system \eqref{eq:pseudo-equation without perturbation} holds. That is, for all solutions $V$, with initial data in $\H^0$, we have the following inequality:
\begin{equation}\label{eq: L^2-observability inequality-pseudo}
    ||V(0)||^2_{\H^0}\leq\F(\|\underline{U}\|_{\H^{s_0}})\int_0^T||\chi_T\varphi_{\omega}V(t)||^2_{\H^0}dt.
\end{equation}
\end{proposition}
Owing to Proposition \ref{prop:L^2 observability-pseudo}, we are in position to prove Proposition \ref{prop:L^2 observability-diagonal}.
\begin{proof}[Proof of Proposition \ref{prop:L^2 observability-diagonal}]
 Let $Y\in C([0,T],\H^0)$ be the solution to the equation \eqref{eq:diagonalized simplified adjoint equation}. Write $Y=V+W$, with $V,W\in C([0,T],\H^0)$ such that
\begin{equation*}
    D_t V=\mathscr{B}^w V,\quad V(0)=Y(0);\quad \left(D_t -\mathscr{B}^w-R^w_{\mathscr{B}}\right)W=R^w_{\mathscr{B}}V,\quad W(0)=0.
\end{equation*}
According to the Duhamel's formula, we obtain
\begin{align*}
    ||V||_{C([0,T],\H^0)}&\lesssim ||V(0)||_{\H^0},\\
    ||W||_{C([0,T],\H^0)}&\lesssim ||R^w_{\mathscr{B}}V||_{L^1([0,T],\H^0)}\lesssim\epsilon_0||V||_{L^1([0,T],\H^0)}\lesssim\epsilon_0T||V(0)||_{\H^0}.
\end{align*}
Therefore, applying Proposition \ref{prop:L^2 observability-pseudo}, we obtain
\begin{align*}
||Y(0)||^2_{\H^0}=||V(0)||^2_{\H^0}&\leq\F(\|\underline{U}\|_{\H^{s_0}})\int_0^T||\chi_T\varphi_{\omega}V(t)||^2_{\H^0}dt\\
                      &\leq \F(\|\underline{U}\|_{\H^{s_0}})( \int_0^T||\chi_T\varphi_{\omega}Y(t)||^2_{\H^0}dt+\int_0^T||\chi_T\varphi_{\omega}W(t)||^2_{\H^0}dt)\\
                      &\leq\F(\|\underline{U}\|_{\H^{s_0}})(\int_0^T||\chi_T\varphi_{\omega}Y(t)||^2_{\H^0}dt+T||W(t)||^2_{C([0,T],\H^0)})\\
                      &\leq\F(\|\underline{U}\|_{\H^{s_0}})(\int_0^T||\chi_T\varphi_{\omega}Y(t)||^2_{\H^0}dt+\epsilon_0^2||V(0)||^2_{\H^0}).
\end{align*}
Consequently, the observability for $Y$ holds for $\epsilon_0$ sufficiently small. 
\end{proof}
And now we only need to prove Proposition \ref{prop:L^2 observability-pseudo} to establish the $L^2-$observability.

\subsection{Semi-classical equation}
In this section, we consider the $L^2-$observability of the pseudo-differential equation \eqref{eq:pseudo-equation without perturbation}.
We derive the equation satisfied by frequency localized quasi-modes of \eqref{eq:pseudo-equation without perturbation}. Since $\omega$ satisfies the geometric control condition, for fixed $T>0$, there exists $\nu>0$ such that 
\begin{equation}\label{eq: propagation speed}
  2\nu T>\min\{L\geq0:\forall(x,\xi)\in \mathbb{T}^d\times\mathbb{S}^{d-1},[x,x+L\xi]\cap\omega\neq\emptyset\},  
\end{equation}
which implies that the wave packets around frequencies with modulus $\geq\nu$ will travel into $\omega$ within time $T$. We define the class of the cutoff functions
\begin{equation*}
    S_{cf}(\nu)=\{\chi\in C^{\infty}_c(\mathbb{R}\backslash\{0\}): 0\leq\chi\leq1,supp\, \chi\subset\{1\leq\frac{|x|}{\nu^2}\leq5\},\chi|_{2\leq\frac{|x|}{\nu^2}\leq4}\equiv1\}.
\end{equation*}
Recall \eqref{matriciozze} and fix $\chi\in S_{cf}(\nu)$, we want to define the frequency localized operator 
\begin{equation*}Q_h(\U):=\chi(h^2\mathscr{B}^w(\underline{U})):=E
\chi(h^2b^w(\underline{U}))
:= E b_h^w(\U),\end{equation*} 
with $h\in(0,1)$ and where $b^w(\underline{U}):=\opw(\lambda(\underline{U};x)|\xi|^2)+\opw(\vec{a}_1(\underline{U};x)\cdot\xi)$; to do so we use the  Helffer-Sj\"ostrand's formula. For more details we refer to \cite{Burq-Gerard-Tzvetkov,Zworski,Zhu}. The Helffer-Sj\"ostrand's formula for functional calculus reads as follows 
\begin{equation}
    b_h^w(\U)=-\frac{1}{\pi}\int_{\C}\bar{\partial}\tilde{\chi}(z)(z-h^2b^w(\underline{U}))^{-1}dz,
\end{equation}
where $\tilde{\chi}$ is an almost analytic extension of $\chi$, i.e. such that, for some $n\in\N$ to be fixed later,
\begin{equation*}
    \tilde{\chi}(z)=\varrho(\Im z)\sum_{k=0}^n\frac{\chi^{k}(\Re z)}{k!}(\ii \Im z)^k,
\end{equation*}
with $\varrho\in C^{\infty}_c(\R)$ such that $\varrho=1$ near zero. 
We set $V_h(t,x)=Q_h V(t,x)$,
we would like to derive the equation satisfied by $V_h$, under the new coordinates $(s,x)$ with $s=h^{-1}t$, which implies that $D_s=h D_t$. First, under the coordinates $(s,x)$, the equation \eqref{eq:pseudo-equation without perturbation} is equivalent to 
\begin{equation}\label{eq:semiclasssical eq for Y}
    h D_s V=h^2\mathscr{B}^w V.
\end{equation}
The equation satisfied by $V_h$ will be given by commuting $Q_h$ with the equation \eqref{eq:semiclasssical eq for Y}. Before doing this, we introduce some properties of the functional calculus for the operator $\mathscr{B}^w(\underline{U})$.

\subsubsection{Semiclassical Functional Calculus for $Q_h$}
Since we have two different coordinates $(t,x)$ and $(s,x)$, we use the notations $[a,b]_s=\{s:a\leq s\leq b\}$ and $[a,b]_t=\{t:a\leq t\leq b\}$ to denote the two different time intervals to avoid ambiguity.
\begin{lemma}\label{lemma: functional calculus}
Let $s\geq s_0>d/2+2$, $T>0$, $\epsilon_0>0$, and $\underline{U}\in \mathscr{C}^{1,s_0}(T,\epsilon_0)$. Then for $\epsilon_0$ sufficiently small, the operator $Q_h\in C([0,T]_t,\mathcal{L}(\H^0,\H^0))$ and 
\begin{equation*}
    \|Q_h-\ophw(B_{\chi})\|_{L^{\infty}([0,T]_t,\mathcal{L}(\H^0,\H^0)}\leq \F(\|\underline{U}\|_{\H^{s_0}})h,
\end{equation*}
where $B_{\chi}(x,\xi)=b_{\chi}E$ and 
\begin{equation}\label{eq: defi of b_chi}
    b_{\chi}=\chi(\lambda(\underline{U};x)|\xi|^2).
\end{equation}
Moreover, 
\begin{equation*}
    \|D_s Q_h\|_{L^{\infty}([0,T]_t,\mathcal{L}(\H^0,\H^0)}\leq \F(\|\underline{U}\|_{\H^{s_0}})h(\epsilon_0+h),
\end{equation*}
and in particular, $Q_h\in W^{1,\infty}([0,T]_t,\mathcal{L}(\H^0,\H^0))$.
\end{lemma}
\begin{proof}
By  definition, 
\begin{align*}
    h^2\mathscr{B}^w(\underline{U})&=h^2\opw(E(\lambda(\underline{U};x))|\xi|^2)+h^2\opw(\diag(\vec{a}_1(\underline{U};x))\cdot\xi)\\
    &=\ophw(E(\lambda(\underline{U};x))|\xi|^2)+h\ophw(\diag(\vec{a}_1(\underline{U};x))\cdot\xi).
\end{align*}
We can argue component-wise. Let $q_0(z,x,\xi)=(z-\lambda(\underline{U};x)|\xi|^2)^{-1}$ for $z\in \C\backslash\R$, then for some $N>0$ {by composition formula for Weyl pseudodifferential operators (we refer to \cite[Theorem 4.11]{Zworski})}, we obtain 
\begin{equation*}
    \ophw(q_0(z))\circ(z-h^2{b}^w)=Id+\O(h|\Im z|^{-N})_{\mathcal{L}(L^2,L^2)}.
\end{equation*}
Combining with $\|(z-h^2{b}^w)^{-1}\|_{\mathcal{L}(L^2,L^2)} \leq |\Im z|^{-1}$, we infer
\begin{equation*}
    (z-h^2b^w)^{-1}=\ophw(q_0(z))+\O(h|\Im z|^{-N-1})_{\mathcal{L}(L^2,L^2)}.
\end{equation*}
We choose $n\geq N$, and we obtain the following:
\begin{align*}
    b^w_h&=-\frac{1}{\pi}\int_{\C}\bar{\partial}\tilde{\chi}(z)(z-h^2 b^w(\underline{U}))^{-1}dz\\
    &=-\frac{1}{\pi}\int_{\C}\bar{\partial}\tilde{\chi}(z)\left(\ophw(q_0(z))+\O(h|\Im z|^{-N-1})_{\mathcal{L}(L^2,L^2)}\right)dz\\
    &=\ophw(-\frac{1}{\pi}\int_{\C}\bar{\partial}\tilde{\chi}(z)(z-\lambda(\underline{U};x))|\xi|^2)^{-1}dz)-\frac{1}{\pi}\int_{\C}\bar{\partial}\tilde{\chi}(z)\O(h|\Im z|^{-N-1})_{\mathcal{L}(L^2,L^2)}dz\\
    &=\ophw(b_{\chi})+\O(h)_{\mathcal{L}(L^2,L^2)}.
\end{align*}
In the penultimate equality, we used Cauchy's integral formula, and $b_{\chi}$ is defined in \eqref{eq: defi of b_chi}. For $D_sb_h^w$, we use the identity
\begin{equation*}
    D_t(z-A)^{-1}=(z-A)^{-1}D_tA(z-A)^{-1}.
\end{equation*}
As a consequence, we obtain 
\begin{align*}
   D_sb_h^w&=-\frac{1}{\pi}\int_{\C}\bar{\partial}\tilde{\chi}(z)D_s(z-h^2{b}^w(\underline{U}))^{-1}dz\\
   &=-\frac{h}{\pi}\int_{\C}\bar{\partial}\tilde{\chi}(z)(z-h^2{b}^w(\underline{U}))^{-1}D_t(h^2{b}^w(\underline{U}))(z-h^2{b}^w(\underline{U}))^{-1}dz.
\end{align*}
Considering $D_t(h^2b^w(\underline{U}))$, 
\begin{align*}
    D_t(h^2b^w(\underline{U}))&=\ophw(D_t(\lambda(\underline{U};x))|\xi|^2)+h\ophw(D_t(\vec{a}_1(\underline{U};x))\cdot\xi).
\end{align*}
Thus, use again Cauchy's integral formula, we obtain 
\begin{align*}
    D_sb_h^w&=-\frac{h}{\pi}\int_{\C}\bar{\partial}\tilde{\chi}(z)(z-h^2{b}^w(\underline{U}))^{-1}\ophw(D_t(\lambda(\underline{U};x))|\xi|^2)(z-h^2b^w(\underline{U}))^{-1}dz\\
    &-\frac{h^2}{\pi}\int_{\C}\bar{\partial}\tilde{\chi}(z)(z-h^2b^w(\underline{U}))^{-1}\ophw(D_t(\vec{a}_1(\underline{U};x))\cdot\xi)(z-h^2b^w(\underline{U}))^{-1}dz\\
    &=\ophw\left(-\frac{h}{\pi}\int_{\C}\bar{\partial}\tilde{\chi}(z)(D_t(\lambda(\underline{U};x))|\xi|^2)q_0(z)^{-2}dz\right)+\O(h^2)_{\mathcal{L}(L^2,L^2)}\\
    &=-h\ophw\left((\partial_z\chi)(b_{\chi})D_t(\lambda(\underline{U};x))|\xi|^2)\right)+\O(h^2)_{\mathcal{L}(L^2,L^2)}\\
    &=\O(h(\epsilon_0+h))_{\mathcal{L}(L^2,L^2)}.
\end{align*}
This concludes the proof.
\end{proof}
\begin{corollary}\label{cor: phase localized estimates}
Let $s\geq s_0>d/2+2$, $T>0$, $\epsilon_0>0$, and $\underline{U}\in \mathscr{C}^{1,s_0}(T,\epsilon_0)$. Let $\chi'\in S_{cf}(\nu)$ and $\chi'\chi=\chi$. Then for $\epsilon_0$ sufficiently small, 
\begin{equation*}
    V_h=Q_h'V_h=\ophw(B_{\chi'})V_h+R_{h} V_h
\end{equation*}
and $\|R_h\|_{L^{\infty}([0,T]_t,\mathcal{L}(\H^0,\H^0))}\leq \F(\|\underline{U}\|_{\H^{s_0}})h$.
\end{corollary}
We shall use this corollary when we considering the semiclassical weak observability for generalized pseudo-differential equations \eqref{eq: generalized pseudo-equation} in Proposition \ref{prop:L^2 semiclassical observability}.
\begin{corollary}
Let $s\geq s_0> d/2+2$, $T>0$, $\epsilon_0>0$, and $\underline{U}\in \mathscr{C}^{1,s_0}(T,\epsilon_0)$. Then for $\epsilon_0$ sufficiently small, $\forall t\in [0,T]$, for $W\in\H^0$, we have
\begin{equation}
    \sum_{h=2^{-j},j\geq 0}||Q_h W||^2_{\H^0}\lesssim ||W||^2_{\H^0}.
\end{equation}
Furthermore, for $j\geq j_0\gg1$, $\forall N>0$, we have
\begin{equation}
    \sum_{h=2^{-j},j\geq j_0}||Q_h W||^2_{\H^0}\lesssim ||W||^2_{\H^0}-||W||^2_{\H^{-N}}.
\end{equation}
\end{corollary}
{We shall apply this corollary to derive the weak observability in Proposition \ref{prop: weak ob}.}
\subsubsection{Equation for $Q_h Y$}
\begin{proposition}
Let $s\geq s_0> d/2+2$, $T>0$, $\epsilon_0>0$, and $\underline{U}\in \mathscr{C}^{1,s_0}(T,\epsilon_0)$. Then for $\epsilon_0$ sufficiently small, $V_h$ satisfies the equation
\begin{equation}\label{eq:semi-classical eq-V_h}
    h D_s Y_h=h^2\mathscr{B}^w Y_h+\tilde{R}_h Y.
\end{equation}
and $\|\tilde{R}_h\|_{L^{\infty}([0,T]_t,\mathcal{L}(\H^0,\H^0))}\leq \F(\|\underline{U}\|_{\H^{s_0}})h^2(\epsilon_0+h)$.
\end{proposition}
\begin{proof}
The remainder  $\tilde{R}_h$ is explicitly computed as follows:
\begin{align*}
    hD_s Y_h&=h^2Q_h D_sY+h(D_s Q_h)Y\\
                  &=h^2Q_h\mathscr{B}^w Y+h(D_t Q_h)Y\\
                  &=h^2\mathscr{B}^w Y_h+h(D_t Q_h)Y.
\end{align*}
Hence, $\tilde{R}_h=h(\partial_t Q_h)$. Using Lemma \ref{lemma: functional calculus}, we obtain that
\begin{equation*}
    \|\tilde{R}_h\|_{L^{\infty}([0,T]_t,\mathcal{L}(\H^0,\H^0)}\leq \F(\|\underline{U}\|_{\H^{s_0}})h^2(\epsilon_0+h).
\end{equation*}
\end{proof} 

\subsection{Semiclassical observability}
\begin{proposition}\label{prop:L^2 semiclassical observability}
Suppose that $\omega$ satisfies the geometric control condition, $s_0>d/2+2$, $T>0$. Then for $\epsilon_0>0$ and $h_0>0$ sufficiently small, for all $0<h<h_0$ and for $\underline{U}\in \mathscr{C}^{1,s_0}(hT,\epsilon_0)$,  the solution $Y_h\in C([0,T],\H^0)$ of the equation 
\begin{equation}\label{eq: generalized pseudo-equation}
    h D_s Y_h-h^2\mathscr{B}^w Y_h=f_h,
\end{equation}
where $f_h\in L^2([0,T],\H^0)$, satisfies the semiclassical observability:
\begin{equation}\label{eq: semiclassical observability inequality}
    ||Y_h(0)||^2_{\H^0}\lesssim\int_0^T||\chi_T\varphi_{\omega}Y_h(t)||^2_{\H^0}dt+h^{-2}||f_h||^2_{L^2([0,T],\H^0)}.
\end{equation}
provided that
\begin{equation}\label{localized frequency restriction}
    Y_h=Q'_h Y_h+R_{h} Y_h
\end{equation}
where $Q'_h=\chi'(h^2\mathscr{B}^w)$, $\chi'$ is defined as in Corollary \ref{cor: phase localized estimates} and $R_h$ satisfies the same estimate as in Corollary \ref{cor: phase localized estimates}. 
\end{proposition}
\begin{proof}
We argue by contradiction. Suppose that \eqref{eq: semiclassical observability inequality} is not true, then there are  sequences
\begin{equation}\label{eq: sq-U_n}
    \epsilon_n>0, h_n>0, \underline{U}_n\in \mathscr{C}^{1,s},(T,\epsilon_n)
\end{equation} and 
\begin{equation}\label{eq: sq-Y_n,f_n}
    (Y_n,f_n)\in C([0,T],\H^0)\times L^2([0,T],\H^0) 
\end{equation}
such that $Y_h=Y_n:=Y_{h_n}$ satisfies \eqref{localized frequency restriction} and is solution to 
\begin{equation}\label{eq: eq-Y_n}
    h_n D_s Y_n-h_n^2\mathscr{B}_n Y_n=f_n,
\end{equation}
with $\mathscr{B}_n=\mathscr{B}^w(\underline{U}_n)$ and $f_h=f_n:=f_{h_n}$. As $n$ tends to infinity, $\epsilon_n=o(1),h_n=o(1)$, 
\begin{equation*}
    ||Y_n(0)||_{\H^0}=1,||\chi_T\varphi_{\omega}Y_h(t)||_{L^2([0,T],\H^0)}=o(1),||f_n||_{L^2([0,T],\H^0)}=o(h_n).
\end{equation*}
By energy estimates, we know that 
\begin{equation*}
    \|Y_n\|_{C^0([0,T],\H^0)}\lesssim \|Y_n(0)\|_{\H^0}+\frac{1}{h_n}\|f_n\|_{L^1([0,T],\H^0)}\lesssim 1.
\end{equation*}
Hence, we obtain a bounded sequence $\{Y_n\}_{n\in\N}$ in $C^0([0,T],\H^0)$. Define the Wigner distribution (see \cite{Macia} for more details) associated with $\{Y_h(s)\}_{h=h_n}$ on the cotangent bundle $T^*\T^d$ by
\begin{equation}\label{eq: Wigner distribution}
    \poscalr{W_Y^h(s,\cdot)}{A}=\poscals{\ophw(A)Y_h(s,\cdot)}{Y_h(s,\cdot)},
\end{equation}
where $A\in M_{2\times2}(C^{\infty}_c(T^*\T^d))$. In particular, at the initial time $s=0$, up to a subsequence, there exists a semiclassical measure $\mu_0\in M_{2\times2}(\mathcal{M}^+(T^*\T^d))$, which is a non-negative Radon measure on $T^*\T^d$, such that,
\begin{equation}\label{eq: measure for initial data}
    \poscalr{W_Y^h(0,\cdot)}{A}=\poscals{\ophw(A)Y_h(0,\cdot)}{Y_h(0,\cdot)}\rightarrow \int_{T^*\T^d}Tr(Ad\mu_0).
\end{equation}
\begin{remark}
For simplicity, when we extract a certain subsequence, we usually do not relabel it and still denote it by its original sequence.  
\end{remark}
Since $||Y_n(0)||_{\H^0}=1$, we know that $Tr\mu_0(T^*\T^d)=1$. We have the following lemma for the Wigner distribution.
\begin{lemma}
Up to an extracting subsequence, there exists a semiclassical measure $\mu$ belonging to space $C([0,T]; M_{2\times2}(\mathcal{M}^+(T^*\T^d)))$ such that for every $\phi\in L^1((0,T))$ and every $A\in M_{2\times2}(C^{\infty}_c(T^*\T^d))$,
\begin{equation}\label{eq:limit-measures}
    \lim_{h\rightarrow0^+}\int_0^T\phi(t)\poscals{\ophw(A)Y_h(t,\cdot)}{Y_h(t,\cdot)}dt=\int_0^T\int_{T^*\T^d}\phi(t)Tr(A(x,\xi)d\mu)dt.
\end{equation}
\end{lemma}
\begin{proof}
According to the Calderon-Vaillancourt Theorem (see Theorems 4.3 and 13.13 in \cite{Zworski}),
\begin{align*}
    |\poscals{\ophw(A)Y_h(s,\cdot)}{Y_h(s,\cdot)}|&\leq \|\ophw(A)\|_{\mathcal{L}(\H^0,\H^0)}\|Y_h(s,\cdot)\|^2_{\H^0}\\
    &\leq C_d(\|A\|_{L^{\infty}}+\O(h)),
\end{align*}
where $C_d$ is a positive constant depending on the dimension. Therefore, 
\begin{equation*}
    |\int_0^T\phi(t)\poscals{\ophw(A)Y_h(t,\cdot)}{Y_h(t,\cdot)}dt|\leq C_d(\|A\|_{L^{\infty}}\|+\O(h))\|\phi\|_{L^1},
\end{equation*}
which implies that, up to a subsequence, $\int_0^T\phi(t)\poscals{\ophw(A)Y_h(t,\cdot)}{Y_h(t,\cdot)}dt$ has a limit as $h\rightarrow0^+$.

Since $C^{\infty}_c(T^*\T^d)$ and $L^1((0,T))$ are separable, after a diagonal process, combining with the Riesz representation theorem, we can extract a subsequence, such that, there exists a Radon measure $\mu$ in the space $L^{\infty}([0,T];\mathcal{M}(T^*\T^d))$ satisfying
\begin{equation*}
    \lim_{h\rightarrow0^+}\int_0^T\phi(t)\poscals{\ophw(A)Y_h(t,\cdot)}{Y_h(t,\cdot)}dt=\int_0^T\int_{T^*\T^d}\phi(t)Tr(A(x,\xi)d\mu)dt.
\end{equation*}
Applying the sharp Gårding inequality, $\mu$ is a non-negative Radon measure. For more details, we refer to the proof of Theorem 5.2 in \cite{Zworski}.

 It remains to prove the continuity with respect to time. We know that  
$\poscals{\ophw(A)Y_h(s,\cdot)}{Y_h(s,\cdot)}$ is uniformly bounded.
In virtue of the fact that the space of non negative Radon measures $\mathcal{M}(T^*\T^d)$ is a complete metric space, we use Ascoli-Arzel\`a theorem.
We prove that  $g_h^A(s)=\poscals{\ophw(A)Y_h(s,\cdot)}{Y_h(s,\cdot)}$ is equicontinuous, in particular that $\partial_sg_h^A(s)$ is bounded in $L^1((0,T))$:
\begin{align*}
    h\partial_sg_h^A(s)&=h\partial_s\poscals{\ophw(A)Y_h(s,\cdot)}{Y_h(s,\cdot)}\\
    &=\poscals{\ii h^2(\mathscr{B}_h^w\ophw(A)-\ophw(A)\mathscr{B}_h^w)Y_h(s,\cdot)}{Y_h(s,\cdot)}\\
    &+\ii\poscals{\ophw(A)Y_h(s,\cdot)}{f_h}- \ii\poscals{\ophw(A)f_h}{Y_h(s,\cdot)}.
\end{align*}
By semiclassical symbolic calculus (\cite[Theorem 4.11]{Zworski})), we know that
\begin{align*}
    h^2(\mathscr{B}_h^w\ophw(A)-\ophw(A)\mathscr{B}_h^w)&=[h^2\mathscr{B}_h^w,\ophw(A)]\\
    &=\O(h)_{L^{\infty}((0,T),\mathcal{L}(\H^0,\H^0))}.
\end{align*}
By Cauchy-Schwarz inequality,
\begin{equation*}
    |\poscals{\ophw(A)f_h}{Y_h(s,\cdot)}|\leq C\|f_h\|_{\H^0}\|Y_h\|_{\H^0}
\end{equation*}
Hence, we obtain that
\begin{align*}
    h\partial_sg_h^A(s)&\leq C\left(h\|Y_h\|^2_{\H^0}+\frac{1}{h}\|f_h\|^2_{\H^0}\right).
\end{align*}
As a consequence, we know that
\begin{equation*}
    \|\partial_sg_h^A\|_{L^1}\leq C\left(\|Y_h\|^2_{\H^0}+\frac{1}{h^2}\|f_h\|^2_{L^1\H^0}\right)\leq C
\end{equation*}
We conclude that there exists a semiclassical measure $\mu\in C([0,T];M_{2\times2}(\mathcal{M}^+(T^*\T^d)))$ such that \eqref{eq:limit-measures} holds true.
\end{proof}

Let us denote $Y_h=\left(\begin{array}{l}
     Y_h^+  \\
     Y_h^-
\end{array}\right)$. 
By the definition of the semi-classical defect measure, we know that it admits the following form:
\begin{equation*}
    \mu=\left(\begin{array}{cc}
         \mu_+&\mu_*  \\
         \Bar{\mu}_*&\mu_- 
    \end{array}\right),
\end{equation*}
where $\mu_{\pm}$ denote respectively the semiclassical measures associated to the sequence $\{Y_h^{\pm}\}$. And moreover, by Cauchy-Schwarz inequality, we also obtain that
\begin{equation}\label{eq:off-diagonal estimate}
    \mu_*\ll\sqrt{\mu_+\mu_-}.
\end{equation}
\begin{remark}
In fact, for every measurable set $E$, we take the identical function $\uno_E$. We notice that $\uno_E^2=\uno_E$. By definition, we obtain
\begin{equation*}
    \poscalr{\uno_E}{\mu_*}=\lim_{h\rightarrow0^+}\poscals{\uno_EY_h^+}{Y_h^-}=\lim_{h\rightarrow0^+}\poscals{\uno_EY_h^+}{\uno_EY_h^-}.
\end{equation*}
Using Cauchy-Schwarz inequality, we obtain
\begin{equation*}
    \lim_{h\rightarrow0^+}\poscals{\uno_EY_h^+}{\uno_EY_h^-}\leq \lim_{h\rightarrow0^+}\poscals{\uno_EY_h^+}{\uno_EY_h^+}^{\frac{1}{2}}\poscals{\uno_EY_h^-}{\uno_EY_h^-}^{\frac{1}{2}}=\poscalr{\uno_E}{\mu_+}^{\frac{1}{2}}\poscalr{\uno_E}{\mu_-}^{\frac{1}{2}}.
\end{equation*}
This implies \eqref{eq:off-diagonal estimate}.
\end{remark}
We are able to conclude the proof with the following proposition.
\begin{proposition}\label{prop: property of measures}
The measure $\mu$ satisfies the following properties:
\begin{enumerate}
    \item $\supp{\mu}\subset \{|\xi|\geq v\}$;
    \item $\mu_{\pm}$ is propagated via the transportation equation 
    \begin{equation}\label{eq: propagation law}
        (\partial_s\pm2\xi\cdot\nabla)\mu_{\pm}=0.
    \end{equation}
\end{enumerate}
\end{proposition}
\begin{proof}[Proof of Proposition \ref{prop: property of measures}]All the sequences we use below are defined in \eqref{eq: sq-U_n}, and \eqref{eq: sq-Y_n,f_n}. The first statement comes from \eqref{localized frequency restriction}. In fact, let $A\in M_{2\times2}(C^{\infty}_c(T^*\T^d))$, from Lemma \ref{lemma: functional calculus}, we know that
\begin{align*}
    Y_h&=Q'_hY_h+o(1)_{C([0,T],\mathcal{L}(\H^0,\H^0))}Y_h\\
       &=\ophw(B_{\chi'})Y_h+o(1)_{C([0,T],\mathcal{L}(\H^0,\H^0))}Y_h\\
       &=\ophw(E\chi'(|\xi|^2))Y_h+o(1)_{C([0,T],\mathcal{L}(\H^0,\H^0))}Y_h\\
       &+\ophw(E(\chi'(\lambda(\underline{U};x)|\xi|^2)-\chi'(|\xi|^2)))Y_h.
\end{align*}
We can bound the term $\ophw(E(\chi'(\lambda(\underline{U};x)|\xi|^2)-\chi'(|\xi|^2)))$ by $\epsilon_0$. Thus, we obtain 
\begin{align*}
    Y_h&=\ophw(\chi'(|\xi|^2))Y_h+o(1)_{C([0,T],\mathcal{L}(\H^0,\H^0))}Y_h.
\end{align*}
Consequently,  for any $A$ and $\phi$ we get 
\begin{align*}
    0&=\lim_{h\rightarrow0^+}\int_0^T\phi(s)\poscals{\ophw(A)(1-\ophw(\chi'(|\xi|^2)))Y_h}{Y_h}ds\\
     &=\int_0^T \int_{T^*\T^d}\phi(s)Tr(A(x,\xi)(1-\chi'(|\xi|^2))d\mu)ds,
\end{align*}
 we conclude that $\forall s\in [0,T]$, $\supp{\mu( s,\cdot)}\subset \supp{\chi'(|\xi|^2)}\subset\{|\xi|\geq\nu\}$.

In order to prove the second statement, we set $A^+=\left(\begin{array}{cc}
     a& 0 \\
     0& 0
\end{array}\right)$ and $A^-=\left(\begin{array}{cc}
     0& 0 \\
     0& a
\end{array}\right)$. For any $\phi\in C^{\infty}_c((0,T))$, we consider
\begin{align*}
   \int_0^T\phi(s)\partial_s\poscals{\ophw(A^{\pm})Y_h}{Y_h}ds&=\int_0^T-\partial_s\phi(s)\poscals{\ophw(A^{\pm})Y_h}{Y_h}ds\\
    &\rightarrow-\int_0^T\int_{T^*\T^d}\partial_s\phi(s)Tr(A^{\pm}d\mu)ds\\
    &=-\int_0^T\int_{T^*\T^d}\partial_s\phi(s)a(x,\xi)d\mu_{\pm}ds\\
    &=\int^T_0\int_{T^*\T^d}\phi(s)a(x,\xi)d\partial_s\mu_{\pm}ds.
\end{align*}
On the other hand, 
\begin{align*}
    \partial_s\poscals{\ophw(A^{\pm})Y_h}{Y_h}&=\poscals{\ophw(A^{\pm})\partial_sY_h}{Y_h}+\poscals{\ophw(A^{\pm})Y_h}{\partial_sY_h}\\
    &=\poscals{\frac{\ii}{h}\ophw(A^{\pm})\left(h^2\mathscr{B}^w(\underline{U}_n)Y_h+f_h\right)}{Y_h}\\
    &+\poscals{\ophw(A^{\pm})Y_h}{\frac{\ii}{h}\left(h^2\mathscr{B}^w(\underline{U}_n)Y_h+f_h\right)}\\
    &=\poscals{\ii h\ophw(A^{\pm})\mathscr{B}^w(\underline{U}_n)Y_h}{Y_h}+\poscals{\frac{\ii}{h}\ophw(A^{\pm})f_h}{Y_h}\\
    &+\poscals{\ophw(A^{\pm})Y_h}{\ii h\mathscr{B}^w(\underline{U}_n)Y_h}+\poscals{\ophw(A^{\pm})Y_h}{\frac{\ii}{h}f_h}\\
    &=\poscals{\ii h\left(\ophw(A^{\pm})\mathscr{B}^w(\underline{U}_n)-\mathscr{B}^w(\underline{U}_n)\ophw(A^{\pm})\right)Y_h}{Y_h}+\O(h).
\end{align*}
Now consider
\begin{align*}
    \ophw(A^{\pm})&\mathscr{B}^w(\underline{U}_n)-\mathscr{B}^w(\underline{U}_n)\ophw(A^{\pm})=\frac{1}{h^2}[\ophw(A^{\pm}),h^2\mathscr{B}^w(\underline{U}_n)]\\
    &=\frac{1}{h^2}[\ophw(A^{\pm}),h^2\mathscr{B}^w(0)]+\frac{1}{h^2}[\ophw(A^{\pm}),h^2\mathscr{B}^w(\underline{U}_n)-h^2\mathscr{B}^w(0)]\\
    &=\frac{\ii}{h}\ophw(\{a,\pm|\xi|^2\})+\frac{1}{h^2}[\ophw(A^{\pm}),h^2\mathscr{B}^w(\underline{U}_n)-h^2\mathscr{B}^w(0)].
\end{align*}
Since $h^2\mathscr{B}^w(\underline{U}_n)-h^2\mathscr{B}^w(0)=h^2\O(\epsilon_n)_{\mathcal{L}(\H^0,\H^0)}$, we  obtain
\begin{align*}
    \ophw(A^{\pm})\mathscr{B}^w(\underline{U}_n)-\mathscr{B}^w(\underline{U}_n)\ophw(A^{\pm})&=\frac{\ii}{h}\ophw(\{a,\pm|\xi|^2\})+\frac{1}{h}\O(\epsilon_n)_{\mathcal{L}(\H^0,\H^0)}.
\end{align*}
As a consequence, 
\begin{align*}
  \partial_s\poscals{\ophw(A^{\pm})Y_h}{Y_h}&=\poscals{\ii h\left(\ophw(A^{\pm})\mathscr{B}^w(\underline{U}_n)-\mathscr{B}^w(\underline{U}_n)\ophw(A^{\pm})\right)Y_h}{Y_h}+\O(h)\\
  &=\poscals{\ii h\left(\frac{\ii}{h}\ophw(\{a,\pm|\xi|^2\})+\frac{1}{h}\O(\epsilon_n)_{\mathcal{L}(\H^0,\H^0)})\right)Y_h}{Y_h}+\O(h)\\
  &=\poscals{\ophw(\{\pm|\xi|^2,a\})Y_h}{Y_h}+\O(h+\epsilon_n),
\end{align*}
which implies that 
\begin{align*}
    \lim_{h\rightarrow0^+}\int_0^T\phi(s)\partial_s\poscals{\ophw(A^{\pm})Y_h}{Y_h}ds&=\lim_{h\rightarrow0^+}\int_0^T\phi(s)\poscals{\ophw(\{\pm|\xi|^2,a\})Y_h}{Y_h}ds\\
    &=\int_0^T\int_{T^*\T^d}\phi(s)(\pm2\xi\cdot\nabla a)d\mu_{\pm} ds.
\end{align*}
Then we conclude that
\begin{equation*}
    \int^T_0\int_{T^*\T^d}\phi(s)a(x,\xi)d\partial_s\mu_{\pm}ds=\int_0^T\int_{T^*\T^d}\phi(s)(\pm2\xi\cdot\nabla a)d\mu_{\pm} ds,
\end{equation*}
which implies that
\begin{equation*}
    (\partial_s\pm2\xi\cdot\nabla)\mu_{\pm}=0.
\end{equation*}
\end{proof}
Now we go back to the proof of Proposition \ref{prop:L^2 semiclassical observability}. 
By our hypothesis that $||\chi_T\varphi_{\omega}Y_h(t)||_{L^2([0,T],\H^0)}=o(1)$, the semiclassical measure $\mu$ vanishes on $(0,T)\times T^*(\omega)$. In particular, this implies that $\mu_{\pm}=0$ on $(0,T)\times T^*(\omega)$. According to the geometric control condition, the propagation law \eqref{eq: propagation law}, and the condition for propagation speed \eqref{eq: propagation speed}, 
\begin{equation*}
    \mu_{\pm}(0)=\mu_{0,\pm}=0.
\end{equation*}
We then conclude by contradiction. By the hypothesis that $\|Y_h(0)\|_{\H^0}=1$, $Tr \mu_0(T^*\T^d)=1=\mu_{0,+}(T^*\T^d)+\mu_{0,-}(T^*\T^d)$, which implies that $Tr \mu_0=\mu_{0,+}+\mu_{0,-}\neq0$.
\end{proof}

\begin{corollary}\label{cor: semiclassicl weak ob}
Suppose that $\omega$ satisfies the geometric control condition, $s>d/2+2$, $T>0$. Then for $\epsilon_0>0$ and $h_0>0$ sufficiently small, for all $0<h<h_0$ and for $\underline{U}\in \mathscr{C}^{0,s}(hT,\epsilon_0)$,  the solution $Y_h\in C([0,T],\H^0)$ to the equation \eqref{eq:semi-classical eq-V_h} satisfies the semiclassical observability:
\begin{equation}\label{eq: semiclassical weak observability inequality}
    h||Y_h(t=khT)||^2_{\H^0}\lesssim\int_{I_k}||\chi_T\varphi_{\omega}Y_h(t)||^2_{\H^0}dt+h^{-2}||R_h Y||^2_{L^2(I_k,\H^0)},
\end{equation}
where $I_k=[khT,(k+1)hT]$.
\end{corollary}
\begin{proof}
The condition \eqref{localized frequency restriction} of Proposition \ref{prop:L^2 semiclassical observability} is verified by Corollary \ref{cor: phase localized estimates}, therefore, by Proposition \ref{prop:L^2 semiclassical observability}, for some $\epsilon_0>0$ and $h_0>0$, uniformly for $k=0,1,\cdots,h^{-1}-1$
\begin{equation*}
    ||Y_h(s=kT)||^2_{\H^0}\lesssim\int_{kT}^{(k+1)T}||\chi_T\varphi_{\omega}Y_h(s)||^2_{\H^0}ds+h^{-2}||R_h Y||^2_{L^2((kT,(k+1)T),\H^0)}.
\end{equation*}
And we conclude by changing variable $s=h^{-1}t$.
\end{proof}

\subsection{Weak Observability}
In this part, we aim to prove the weak observability for high frequencies. 
\begin{proposition}\label{prop: weak ob}
Suppose that $\omega$ satisfies the geometric control condition, $s>d/2+2$, $T>0$. Then for $\epsilon_0>0$ sufficiently small,  for $\underline{U}\in \mathscr{C}^{0,s}(T,\epsilon_0)$,  the solution $Y\in C([0,T],\H^0)$ satisfies the weak observability:
\begin{equation}\label{eq: weak observability inequality}
    ||Y(0)||^2_{\H^0}\lesssim\int_0^T||\chi_T\varphi_{\omega}Y(t)||^2_{\H^0}dt+||Y(0)||^2_{\H^{-N}}.
\end{equation}
\end{proposition}
\begin{proof}
The idea is to sum up \eqref{eq: semiclassical weak observability inequality} to obtain observability on the whole time interval $[0,T]$ and then use Littlewood-Paley's theory to conclude. For the left hand side of \eqref{eq: semiclassical weak observability inequality}, the term $||Y_h(t=khT)||^2_{\H^0}$ is uniformly bounded from below by $||Y_h(0)||^2_{\H^0}$, according to the energy estimates. Now, let us define the translation operator $\tau_{a}$ by $\tau_{a} G(\cdot)=G(a-\cdot)$. The equation for $\tau_{a} Y_h$ is
\begin{equation}\label{eq: time reversed linear equation}
    -h D_s \tau_{a}Y_h=h^2\mathscr{B}^w(\tau_{a}\underline{U}) \tau_{a}Y_h+\tilde{R}_h\tau_{a}Y.
\end{equation}
We have the following energy estimate
\begin{align*}
    h\partial_s\|\tau_{a}Y_h(s)\|^2_{\H^0}&=\ii\poscals{hD_s\tau_{a}Y_h(s)}{\tau_{a}Y_h(s)}-\ii\poscals{\tau_{a}Y_h(s)}{hD_s\tau_{a}Y_h(s)}\\
    &=-\ii\poscals{\tilde{R}_h \tau_{a}Y}{\tau_{a}Y_h(s)}
    +\ii\poscals{\tau_{a}Y_h(s)}{\tilde{R}_h \tau_{a}Y}\\
    &\leq 2\|\tilde{R}_h \tau_{a}Y\|_{\H^0}\|\tau_{a}Y_h\|_{\H^0}\leq h^{-2}\|\tilde{R}_h \tau_{a}Y\|^2_{\H^0}+h^2\|\tau_{a}Y_h\|^2_{\H^0},
\end{align*}
    from which we deduce that
\begin{equation*}
    \partial_s\|\tau_{a}Y_h(s)\|^2_{\H^0}-h\|\tau_{a}Y_h\|^2_{\H^0}\leq h^{-3}\|\tilde{R}_h \tau_{a}Y\|^2_{\H^0}.
\end{equation*}
As a consequence, 
\begin{equation*}
    \partial_s(e^{-sh}\|\tau_{a}Y_h(s)\|^2_{\H^0})\leq e^{-sh}h^{-3}\|\tilde{R}_h \tau_{a}Y\|^2_{\H^0}.
\end{equation*}
Notice that $hs$ is bounded for $s\in[0,h^{-1}T]$, by Newton-Leibniz's formula,
\begin{align*}
    e^{-sh}\|\tau_{a}Y_h(s)\|^2_{\H^0}&\leq\|\tau_{a}Y_h(0)\|^2_{\H^0} +\int_0^se^{-\sigma h}h^{-3}\|\tilde{R}_h \tau_{a}Y\|^2_{\H^0}d\sigma\\
    &\lesssim \|\tau_{a}Y_h(0)\|^2_{\H^0}+h^{-3}\|\tilde{R}_h \tau_{a}Y\|^2_{L^2([0,s]_s,\H^0)}.
\end{align*} 
By choosing $a=s$ we eventually obtain
\begin{equation*}
   \|Y_h(0)\|^2_{\H^0}\lesssim \|Y_h(s)\|^2_{\H^0}+h^{-3}\|\tilde{R}_h Y\|^2_{L^2([0,h^{-1}T]_s,\H^0)},
\end{equation*}
or equivalently, for $t\in [0,T]$
\begin{equation*}
    h\|Y_h(0)\|^2_{\H^0}\lesssim h\|Y_h(t)\|^2_{\H^0}+h^{-3}\|\tilde{R}_h Y\|^2_{L^2([0,T]_t,\H^0)}.
\end{equation*}
We set $t=khT$ for $h=2^{-j}$ with $j\in2\N$ sufficiently large, and $k=0,1,\cdots,h^{-1}-1$, we use \eqref{eq: semiclassical weak observability inequality} by absorbing $h^{-2}||R_h Y||^2_{L^2(I_k,\H^0)}$ into $h^{-3}\|\tilde{R}_h Y\|^2_{L^2([0,T]_t,\H^0)}$, we have 
\begin{equation*}
    h\|Y_h(0)\|^2_{\H^0}\lesssim \int_{I_k}||\chi_T\varphi_{\omega}Y_h(t)||^2_{\H^0}dt+h^{-3}\|\tilde{R}_h Y\|^2_{L^2([0,T]_t,\H^0)}.
\end{equation*}
Summing  up for $k=0,1,\cdots,h^{-1}-1$,
\begin{equation}
    \|Y_h(0)\|^2_{\H^0}\lesssim \int_0^T||\chi_T\varphi_{\omega}Y_h(t)||^2_{\H^0}dt+h^{-4}\|\tilde{R}_h Y\|^2_{L^2([0,T]_t,\H^0)}.
\end{equation}
Using Littlewood-Paley's theory,
\begin{equation}
    \|Y(0)\|^2_{\H^0}\lesssim \int_0^T||\chi_T\varphi_{\omega}Y(t)||^2_{\H^0}dt+(\epsilon_0^2+h_0^2)||Y(t)||^2_{L^2([0,T],\H^0)}+\| Y(0)\|^2_{\H^{-N}}.
\end{equation}
Since $||Y(t)||^2_{L^2([0,T],\H^0)}\lesssim \|Y(0)\|^2_{\H^0}$, we can absorb $(\epsilon_0^2+h_0^2)||Y(0)||^2_{\H^0}$ in the left hand side when $\epsilon_0$ sufficiently small and $h_0=2^{-j_0}$ sufficiently small.
\end{proof}

\subsection{Unique continuation and strong observability}\label{UCSO}
In this part, we aim to remove the remainder $\|Y(0)\|_{\H^{-N}}^2$  in \eqref{eq: weak observability inequality} by using a uniqueness-compactness argument. We then finish the proof of Proposition \ref{prop:L^2 observability-pseudo}.
\begin{proof}[Proof of Proposition \ref{prop:L^2 observability-pseudo}]
We argue by contradiction. Suppose that the strong observability \eqref{eq: L^2-observability inequality-pseudo} is false. Then there exists a sequence $\{\epsilon_n,\underline{U}_n,Y_n\}_{n\in\N}$, with $\epsilon_n>0$, $\underline{U}_n\in \mathscr{C}^{1,s}(T,\epsilon_n)$ and $Y_n\in C([0,T],\H^0)$ satisfying the equation
\begin{equation}\label{eq: adjoint eq in contradiction}
    D_t Y_n=\mathscr{B}^w(\underline{U}_n) Y_n,
\end{equation}\label{hpabs}
such that, as $n\rightarrow\infty$, 
\begin{equation}
    \epsilon_n=o(1),\quad ||Y_n(0)||_{\H^0}=1,\quad \int_0^T||\chi_T\varphi_{\omega}Y_n(t)||^2_{\H^0}dt=o(1).
\end{equation}
By energy estimates, we know that $\{Y_n\}$ is a bounded sequence in $C([0,T],\H^0)$ and $\{\partial_t Y_n\}$ is also bounded in $L^{\infty}([0,T],\H^{-2})$. Therefore, by Arzel\`a-Ascoli's theorem, we may extract a subsequence such that
\begin{align*}
    &Y_n\rightarrow Y \text{ strongly in } C([0,T],\H^{-2}),\\
    &Y_n\rightharpoonup Y \text{ weakly in } L^2([0,T],\H^0),\\
    &Y_n(0)\rightharpoonup Y(0) \text{ weakly in } \H^0.
\end{align*}
Notice that $\underline{U}_n\in \mathscr{C}^{1,s}(T,\epsilon_n)$ implies that $\underline{U}_n\rightarrow 0$ in $C([0,T],\H^{s})$. Hence, by means of a classical interpolation argument, $\mathscr{B}^w(\underline{U}_n)Y_n\rightarrow \mathscr{B}^w(0)Y$ strongly in $C([0,T],\H^{-2-\epsilon})$ for any $\epsilon>0$. Now passing to the limit as $n\rightarrow\infty$ of \eqref{eq: adjoint eq in contradiction} in the sense of distribution, we obtain that $Y\in C([0,T],\H^{0})$ because it satisfies the limit equation
\begin{equation}\label{eq: limit eq}
    D_t Y=\mathscr{B}^w(0) Y.
\end{equation}
By weak convergence, we know that $\chi_T\varphi_{\omega}Y_n\rightharpoonup \chi_T\varphi_{\omega}Y$ in $ L^2([0,T],\H^0)$. In virtue of \eqref{hpabs}, this implies that
\begin{equation*}
    \int_0^T||\chi_T\varphi_{\omega}Y(t)||^2_{\H^0}dt\leq\liminf_{n\rightarrow\infty}\int_0^T||\chi_T\varphi_{\omega}Y_n(t)||^2_{\H^0}dt=0.
\end{equation*}
By the weak observability,\eqref{hpabs} and the compact injection theorem, we know that 
\begin{align*}
    \|Y(0)\|^2_{\H^{-N}}&=\lim_{n\rightarrow\infty}\|Y_n(0)\|^2_{\H^{-N}}\\
    &\gtrsim \limsup_{n\rightarrow\infty}\left(\|Y_n(0)\|^2_{\H^0}-\int_0^T||\chi_T\varphi_{\omega}Y_n(t)||^2_{\H^0}dt\right)=1.
\end{align*}
To conclude, it suffices to prove the unique continuation property of \eqref{eq: limit eq} and obtain a contradiction. This relies on the following lemma.
\begin{lemma}\label{azz}
Under the hypothesis of Proposition \ref{prop:L^2 observability-pseudo}, suppose that $Y\in C([0,T],\H^{0})$ satisfies \eqref{eq: limit eq} and $Y|_{I\times\omega}=0$ for some interval $I\subset[0,T]$ with non-empty interior, then $Y\equiv0$.
\end{lemma}
By applying such a Lemma one concludes the proof of Proposition \ref{prop:L^2 observability-pseudo}.
\end{proof}
\begin{proof}[Proof of Lemma \ref{azz}]
Assume $I=[0,T]$. For any $\delta\in[0,T)$, we are able to define
the space
\begin{equation*}
    \mathscr{N}_{\delta}=\{Y_0\in\H^0:e^{-\ii t\mathscr{B}^w(0)}Y_0|_{[0,T-\delta]\times\omega}=0\},
\end{equation*}
where $e^{-\ii t\mathscr{B}^w(0)}Y_0$ denotes the solution to the equation \eqref{eq: limit eq} with initial data $Y_0$. According to the weak observability \eqref{eq: weak observability inequality} with $\epsilon_0=0,\underline{U}=0,N>0$, for $Y_0\in \mathscr{N}_{\delta}$, we know that
\begin{equation}\label{eq: compactness inequality}
    ||Y_0||^2_{\H^0}\leq C_{T-\delta}||Y_0||^2_{\H^{-N}},
\end{equation}
where the constant $C_{T-\delta}$ is uniformly bounded as long as $T-\delta$ stays away from $0$. This implies that the closed unit ball in $(\mathscr{N}_{\delta},\|\cdot\|_{\H^0})$ is compact. Hence,
\begin{equation*}
    dim \mathscr{N}_{\delta}<\infty,\quad \forall \delta\in[0,T).
\end{equation*}
Let $\delta<\delta'$. By definition of $\mathscr{N}_{\delta}$, $\mathscr{N}_{\delta}\subset\mathscr{N}_{\delta'}$.  
If $dim \mathscr{N}_{0}=0$, then the proof is closed. Otherwise, there exists a $\delta_0>0$ such that
\begin{equation*}
    \mathscr{N}_{\delta}=\mathscr{N}_{\delta_0}, \text{ for } 0< \delta\leq\delta_0.
\end{equation*}
Set $\epsilon<\delta_0$, $Y_0\in\mathscr{N}_{\delta_0}$ and $Y(t)=e^{-\ii t\mathscr{B}^w(0)}Y_0$. Then $Y(t)=e^{-\ii (t-\epsilon)\mathscr{B}^w(0)}Y(\epsilon)$. It is easy to see that $Y(\epsilon)\in\mathscr{N}_{\delta_0}$. Since $\mathscr{N}_{\delta_0}$ is linear vector space, we have that $\frac{1}{\ii\epsilon}(Y(\epsilon)-Y(0))\in\mathscr{N}_{\delta_0}$. By the inequality \eqref{eq: compactness inequality} with $N=2$, 
\begin{align*}
    ||\frac{1}{\ii\epsilon}(Y(\epsilon)-Y(0))||^2_{\H^0}&\leq C_{T-\delta_0}||\frac{1}{\ii\epsilon}(Y(\epsilon)-Y(0))||^2_{\H^{-2}}\\
    &\lesssim \sup_{t\in[0,\epsilon]}||D_tY(t)||^2_{\H^{-2}}\lesssim \sup_{t\in[0,\epsilon]}||\mathscr{B}^w(0)Y(t)||^2_{\H^{-2}}\\
    &\lesssim \sup_{t\in[0,\epsilon]}||Y(t)||^2_{\H^{0}}\lesssim ||Y_0||^2_{\H^{0}}.
\end{align*}
$\{\frac{1}{\ii\epsilon}(Y(\epsilon)-Y(0))\}_{\epsilon\in(0,\delta)}$ is a bounded family in $(\mathscr{N}_{\delta_0},\|\cdot\|_{\H^0})$. Then we can extract a subsequence $\{\frac{1}{\ii\epsilon_n}(Y(\epsilon_n)-Y(0))\}_{\epsilon_n}$ such that as $n\rightarrow\infty$
\begin{equation*}
    \epsilon_n\rightarrow0,\quad \frac{1}{\ii\epsilon_n}(Y(\epsilon_n)-Y(0))\rightarrow D_sY|_{s=0}=\mathscr{B}^w(0) Y_0 \text{ strongly in }(\mathscr{N}_{\delta_0},\|\cdot\|_{\H^0}).
\end{equation*}
Therefore, we have a well-defined linear map on $(\mathscr{N}_{\delta_0},\|\cdot\|_{\H^0})$,
$$Y_0\mapsto \mathscr{B}^w(0) Y_0, $$
which admits an eigenfunction. Let $Z_0=\left(\begin{array}{c}
    Z_0^+  \\
    Z_0^- 
\end{array}\right)\neq0$ be the eigenfunction, with $\mathscr{B}^w(0) Z_0=\lambda Z_0$ for some $\lambda\in\C$. By the definition of $\mathscr{B}^w(0)$, we have
\begin{equation*}
\left\{
\begin{array}{l}
     -\Delta Z_0^+=\lambda Z_0^+\\
    \Delta Z_0^-=\lambda Z_0^-.
\end{array}
\right.
\end{equation*}
So $Z_0^{\pm}$ are both analytic. Since $Z_0$ vanishes on $\omega$, $Z_0\equiv0$. We conclude that $dim\mathscr{N}_{\delta_0}=0$. 
\end{proof}

\section{$H^s$ null controllability for the linear system}\label{sec: H^s linear control}
\subsection{Regularity of the control operator $\mathscr{L}$}
Recall the HUM operator defined in \eqref{defi: HUM op}, our control operator $\mathscr{L}=\mathscr{L}(\underline{U})$ is defined as 
\begin{equation}\label{eq: control operator}
    \mathscr{L}(\underline{U})=-\ii\chi_T\varphi_{\omega}E\mathcal{S}\mathscr{K}^{-1}.
\end{equation}
This section is devoted to showing the following proposition.
\begin{proposition}\label{prop: regularity of control operator}
Suppose that $\omega$ satisfies the geometric control condition, $s>d/2+2$, $T>0$. Then for $\epsilon_0>0$ sufficiently small,  for $\underline{U}\in \mathscr{C}^{1,s}(T,\epsilon_0)$, the control operator $\mathscr{L}|_{\H^{\s}}=-\ii\chi_T\varphi_{\omega}E\mathcal{S}\mathscr{K}^{-1}$ defines a bounded linear operator from $\H^{\s}$ to $C([0,T],\H^{\s})$, $\forall \s\geq0$ such that
\begin{equation*}
    \|\mathscr{L}|_{\H^{\s}}\|_{\mathcal{L}(\H^{\s},C([0,T],\H^{\s}))}\lesssim 1.
\end{equation*}
\end{proposition}
This proposition is a consequence of Proposition \ref{prop: HUM H^s isomorphism}. So we first analyze the properties of the HUM operator $\mathscr{K}$.
\begin{lemma}
Under the hypothesis of Proposition \ref{prop: regularity of control operator}, for $\epsilon_0$ sufficiently small, $\mathscr{K}|_{\H^{\sigma}}$ sends $\H^{\sigma}(\T^d)$ to itself.
\end{lemma}
\begin{proof}
It is sufficient  to show that the range operator $\mathcal{R}$ and the solution operator $\mathcal{S}$ satisfy 
\begin{equation*}
   \mathcal{R}:L^2([0,T],\H^{\sigma})\rightarrow \H^{\sigma},\quad \mathcal{S}:\H^{\sigma}\rightarrow C([0,T],\H^{\sigma}),
\end{equation*}
and
\begin{equation*}
    \|\mathcal{R}\|_{\mathcal{L}(L^2([0,T],\H^{\sigma}); \H^{\sigma})}\lesssim 1,\quad \|\mathcal{S}\|_{\mathcal{L}(\H^{\sigma}; L^2([0,T],\H^{\sigma}))}\lesssim 1.
\end{equation*}
These estimates are given by Corollary \ref{estimates-near-id}.
\end{proof}
Recall the equivalent norm on $\H^{\sigma}$ given in \eqref{modified} and \eqref{big-Lambda}, which is adapted to the equation \eqref{paradiag}.
We defined the HUM operator associated to the equation in the form of \eqref{NLSparalin}. As we showed in Proposition \ref{diago}, the two equations are equivalent under the invertible linear map $\Phi(\underline{U})$ from $\H^s$ to $\H^s$ for any $s\geq s_0>d/2+2$. We introduce the following  operator $\tilde{\Lambda}_h^{\s}$
\begin{equation}\label{eq: modified weight operator}
\begin{aligned}
&\tilde{\Lambda}^{\s}_h=\Phi(\underline{U})^{-1}\Lambda^{\sigma}_h\Phi(\underline{U}),\\
&\Lambda_h^{\s}:=\opbw\left((1+h^{2}|\xi|^{2}\lambda(\underline{U};x)\right)^{\s/2})
\end{aligned}
\end{equation}
for $h\in[0,1]$. For technical reasons, we work in semiclassical Sobolev spaces $\H^{\s}_h=\left(\H^{\s}(\T^d),\|\cdot\|_{\H^{\s}_h}\right)$, endowed with a scalar product
\begin{equation*}
    \poscalr{U}{V}_{\H^{\s}_h}=\poscals{\langle hD_x\rangle^{\s}U}{\langle hD_x\rangle^{\s}V}.
\end{equation*}
For each fixed $h$, $\H^{\s}_h$ and $\H^{\s}$ are isomorphic as Banach spaces with equivalent norms, even though not uniformly in $h$.
\begin{lemma}\label{lem: regularity of the weight op}
Under the hypothesis of Proposition \ref{prop: regularity of control operator}, for $\s\geq0$, and for $\epsilon_0$ sufficiently small, the operator $\tilde{\Lambda}_h^{\s}: \H^{\s}_h(\T^d)\rightarrow\H^0(\T^d)$ is invertible. Moreover, we have the following estimate, uniformly in $h$
\begin{equation}\label{eq: weight operator equiv}
    \tilde{\Lambda}^{\s}_h-\langle hD_x\rangle^{\s}=\O(\epsilon_0)_{L^{\infty}([0,T],\mathcal{L}(\H_h^{\s},\H^0))},
\end{equation}
which in particular implies the norm equivalence, uniformly for $t_0\in[0,T)$,
\begin{equation}\label{eq: norm equiv}
    \|\cdot\|_{\H_h^{\s}}\sim\| \tilde{\Lambda}_h^{\s}|_{t=t_0}\cdot\|_{\H^0}
\end{equation}
\end{lemma}
\begin{proof}
As we already proved in Proposition \ref{diago}, 
\begin{equation*}
    \|\Phi(\underline{U})^{-1}\|_{L^{\infty}([0,T],\mathcal{L}(\H^{\s},\H^{\s}))}+\|\Phi(\underline{U})\|_{L^{\infty}([0,T],\mathcal{L}(\H^{\s},\H^{\s}))}\lesssim (1+\|\underline{U}\|_{\H^{s_0}})\lesssim 1.
\end{equation*}
Recalling \eqref{eq: modified weight operator}, we only need to prove that $\Lambda_h^{\s}$ is bounded and invertible from $\H^{\s}(\T^d)$ to $\H^0(\T^d)$. Write
\begin{equation*}
  \Lambda_h^{\s} =(1+\Theta_h(\underline{U}))\langle h D_x\rangle^{\s},
\end{equation*}
with $\Theta_h(\underline{U})=(\Lambda_h^{\sigma}-\langle h D_x\rangle^{\sigma}) \langle hD_x\rangle^{-\sigma}$.
According to Theorem \ref{azione}, 
$\Theta_h(\underline{U})=\O(\epsilon_0)_{\mathcal{L}(\H^{0},\H^{0})}$,
uniformly in $h$. Therefore, for $\epsilon_0$ sufficiently small, we know that
\begin{equation*}
    Id+\Theta_h(\underline{U}):\H^0\rightarrow\H^0
\end{equation*}
is invertible, which gives the estimate \eqref{eq: weight operator equiv}. The norm equivalence follows as  the operators
$$\left(\opbw((1+h^2|\xi|^{2})^{\s/2})\right)^{-1}\langle h D_x\rangle^{\s},\quad \opbw((1+h^2|\xi|^{2})^{\sigma/2})\langle h D_x\rangle^{-\s}$$ are both bounded on $\H^0$ since they are both Fourier multipliers which are bounded independently of $h$.
\end{proof}
\begin{remark}
We denote the inverse of $\Lambda^{\s}_h$ simply by $\Lambda^{-\s}_h$, same thing for $\tilde{\Lambda}_h^{\s}$.
\end{remark}
Then we consider a key commutator estimate in the following lemma.
\begin{lemma}\label{lem: HUM commutator estimate}
Suppose that $\omega$ satisfies the geometric control condition, $s>d/2+2$, $T>0$. Then for $\epsilon_0>0$, $h>0$ sufficiently small,  for $\underline{U}\in \mathscr{C}^{1,s}(T,\epsilon_0)$, and for $\sigma\geq1$, the following commutator estimate holds:
\begin{equation}\label{eq: commutator estimate}
    \big[\mathscr{K},\tilde{\Lambda}_{h,0}^{\s}\big]\tilde{\Lambda}_{h,0}^{-\s}=\O(\epsilon_0+h)_{\mathscr{L}(\H^0,\H^0)},
\end{equation}
where we denoted $\tilde{\Lambda}_{h,0}^{\s}:=\tilde{\Lambda}_h^{\s}|_{t=0}$.
\end{lemma}
\begin{proof}
First, by the definition of HUM operator $\mathscr{K}$, we decompose the commutator into three parts as follows
\begin{align*}
    \big[\mathscr{K},\Tilde{\Lambda} _{h,0}^{\s}\big]\Tilde{\Lambda} _{h,0}^{-\s}&=\big[-\mathcal{R}\chi_T^2\varphi_{\omega}^2\mathcal{S},\Tilde{\Lambda} _{h,0}^{\s}\big]\Tilde{\Lambda} _{h,0}^{-\s}\\
    &=\left(-\mathcal{R}\Tilde{\Lambda} _h^{\s}+\Tilde{\Lambda} _{h,0}^{\s}\mathcal{R}\right)\chi_T^2\varphi_{\omega}^2\mathcal{S}\Tilde{\Lambda} _{h,0}^{-\s}-\mathcal{R}\left(\chi_T^2\varphi_{\omega}^2\Tilde{\Lambda} _h^{\s}-\Tilde{\Lambda} _h^{\s}\chi_T^2\varphi_{\omega}^2\right)\mathcal{S}\Tilde{\Lambda} _{h,0}^{-\s}\\
    &-\mathcal{R}\chi_T^2\varphi_{\omega}^2\left(\mathcal{S}\Tilde{\Lambda} _{h,0}^{\s}-\Tilde{\Lambda} _h^{\s}\mathcal{S}\right)\Tilde{\Lambda} _{h,0}^{-\s}.
\end{align*}
By Lemma \ref{lem: regularity of the weight op}, we know that 
\begin{equation*}
    \Tilde{\Lambda}_{h,0}^{-\s}=\O(1)_{\mathcal{L}(\H^{0},\H^{\s})},\quad  \Tilde\Lambda_{h,0}^{-\s}=h\O(1)_{\mathcal{L}(\H^{0},\H^{\s-1})}.
\end{equation*}
Then the estimate \eqref{eq: commutator estimate} is a consequence of the following:
\begin{align}
    &\|-\mathcal{R}\Tilde{\Lambda}^{\s}_{h,0}+\Tilde{\Lambda}_{h,0}^{\s}\mathcal{R}\|_{\mathcal{L}(L^2([0,T],\H^{\s}),\H^{0})}\lesssim\epsilon_0,\label{est-1}\\ 
    &\|\chi_T^2\varphi_{\omega}^2\Tilde{\Lambda}_{h,0}^{\s}-\Tilde{\Lambda}_{h,0}^{\s}\chi_T^2\varphi_{\omega}^2\|_{C([0,T],\mathcal{L}(\H^{\s-1},\H^{0}))}\lesssim1,\label{est-2}\\
    &\|\mathcal{S}\Tilde{\Lambda}_{h,0}^{\s}-\Tilde{\Lambda}_{h,0}^{\s}\mathcal{S}\|_{\mathcal{L}(\H^{\s},C([0,T],\H^{0}))}\lesssim\epsilon_0.\label{est-3}
\end{align}
We prove \eqref{est-1}.
Let $G\in L^2([0,T],\H^{\s})$, and let $Y,\, Z\in C([0,T],\H^{\s})$  be solutions to the two equations respectively.
\begin{equation*}
    \partial_t Y=\ii \mathscr{A}(\underline{U})Y+G,\quad Y(T)=0,\quad \partial_t Z=\ii \mathscr{A}(\underline{U})Z+\Tilde{\Lambda}_{h,0}^{\s}G,\quad Z(T)=0.
\end{equation*}
Then we set $W=Z-\Tilde{\Lambda}_{h,0}^{\s}Y\in C([0,T],\H^{0})$. Moreover, $W$ satisfies the equation
\begin{align*}
     \partial_t W&=\partial_t(Z-\Tilde{\Lambda}_{h,0} ^{\s}Y)\\
                 &=\ii \mathscr{A}(\underline{U})Z+\Tilde{\Lambda}_{h,0} ^{\s}G-\ii \Tilde{\Lambda}_{h,0} ^{\s}\mathscr{A}(\underline{U})Y-\Tilde{\Lambda}_{h,0} ^{\s}G-(\partial_t\Tilde{\Lambda}_{h,0} ^{\s})Y\\
                 &=\ii \mathscr{A}(\underline{U})W+\ii \mathscr{A}(\underline{U})\Tilde{\Lambda}_{h,0} ^{\s}Y-\ii \Tilde{\Lambda}_{h,0} ^{\s}\mathscr{A}(\underline{U})Y-(\partial_t\Tilde{\Lambda}_{h,0} ^{\s})Y,
\end{align*}
i.e. $\partial_t W=\ii \mathscr{A}(\underline{U})W+\ii[\mathscr{A}(\underline{U}),\Tilde{\Lambda}_{h,0} ^{\s}]Y-(\partial_t\Tilde{\Lambda}_{h,0} ^{\s})Y$ with target data $W(T)=0$. Abusing with notation we denote $\Tilde{\Lambda}^{\s}_{h,0}=\diag(\Tilde{\Lambda}^{\s}_{h,0})$. We see the term $\ii[\mathscr{A}(\underline{U}),\Tilde{\Lambda}_{h,0} ^{\s}]Y-(\partial_t\Tilde{\Lambda}_{h,0} ^{\s})Y$ as a source term and we are able to show that it is $\O(\epsilon_0)_{\H^0}$. 
By
\begin{align*}
    \mathscr{A}(\underline{U})\Tilde{\Lambda}_{h,0} ^{\s}&=\mathscr{A}(\underline{U})\Phi(\underline{U})^{-1}\Lambda^{\s}\Phi(\underline{U}),\\
    \Tilde{\Lambda}_{h,0} ^{\s}\mathscr{A}(\underline{U})&=\Phi(\underline{U})^{-1}\Lambda^{\s}\Phi(\underline{U})\mathscr{A}(\underline{U}).
\end{align*}
Therefore, the commutator $[\mathscr{A}(\underline{U}),\Tilde{\Lambda}_{h,0} ^{\s}]$ is given by 
\begin{align*}
    \big[\mathscr{A},\Tilde{\Lambda}_{h,0} ^{\s}\big]&=\mathscr{A}\Phi^{-1}\Lambda^{\s}\Phi-\Phi^{-1}\Lambda^{\s}\Phi\mathscr{A}\\
    &=\Phi^{-1}\left(\Phi\mathscr{A}\Phi^{-1}\Lambda^{\s}-\Lambda^{\s}\Phi\mathscr{A}\Phi^{-1}\right)\Phi.
\end{align*}
According to Proposition \ref{diago}, we know that $\Phi\mathscr{A}\Phi^{-1}=E\opbw((\lambda|\xi|^2))+\Tilde{Q}(\underline{U})$, with $\|\Tilde{Q}\|_{\mathscr{L}(\H^{\s},\H^{\s-1})}\lesssim\epsilon_0$ being a matrices of symbols of order one plus a bounded remainder . Then we simplify the commutator 
\begin{equation*}
   \Big[E\opbw(\lambda|\xi|^2),(\Tilde{\Lambda}_{h,0}^{\s})\Big]=\O(\epsilon_0)_{\mathscr{L}(\H^{\s},\H^{0})},
\end{equation*}
because the Poisson bracket between the symbols is equal to zero. Hence,
\begin{align*}
    \big[\mathscr{A},\Tilde{\Lambda}_{h,0} ^{\s}\big]&=\Phi^{-1}\big[\Tilde{Q},\Lambda^{\s}\big]\Phi+\O(\epsilon_0)_{\mathscr{L}(\H^{\s},\H^{0})},\\
    &=\O(\epsilon_0)_{\mathscr{L}(\H^{\s},\H^{0})},
\end{align*}
where the lower order term $[\Tilde{Q},\opbw(\lambda^{\frac{\sigma}{2}}|\xi|^{\sigma})]=\O(\epsilon_0)_{\mathscr{L}(\H^{\s},\H^{0})}$ is given by symbol calculus. On the other hand $ \partial_t\Tilde{\Lambda}_{h,0} ^{\s}=\O(\epsilon_0)_{\mathscr{L}(\H^{\s},\H^{0})}.$
We just repeat the same proof as we already did in Proposition \ref{diago} and Proposition \ref{linear existence}. Now we conclude using the energy estimates for the equation
\begin{equation*}
    \partial_t W=\ii \mathscr{A}(\underline{U})W+\O(\epsilon_0)_{\mathscr{L}(\H^{\s},\H^{0})}Y,\quad W(T)=0.
\end{equation*}
To be more specific,
\begin{align*}
\|-\mathcal{R}\Tilde{\Lambda}_{h,0} ^{\s}+\Tilde{\Lambda}_{h,0} ^{\s}|\mathcal{R}G\|_{\H^{0}}&=\|W(0)\|_{\H^{0}}\\
&\leq \|\ii[\mathscr{A}(\underline{U}),\Tilde{\Lambda}_{h,0} ^{\s}]Y-(\partial_t\Tilde{\Lambda}_{h,0} ^{\s})Y\|_{L^1([0,T],\H^{0})}\\
&\leq \|[\mathscr{A}(\underline{U}),\Tilde{\Lambda}_{h,0} ^{\s}]-(\partial_t\Tilde{\Lambda}_{h,0} ^{\s})\|_{\mathscr{L}(\H^{\s},\H^{0})}\|Y\|_{L^1([0,T],\H^{\s})}\\
&\lesssim \epsilon_0\|Y\|_{L^1([0,T],\H^{\s})}\lesssim \epsilon_0\|G\|_{L^2([0,T],\H^{\s})}.
\end{align*}
We now prove \eqref{est-3}.
Let $Y_0\in\H^{\s}$, and let $Y\in C([0,T],\H^{\s})$, $Z\in C([0,T],\H^{0})$ be solutions to the following equations, respectively.
\begin{equation*}
\partial_t Y=\ii \mathscr{A}(\underline{U})Y,\quad Y(0)=Y_0,\quad \partial_t Z=\ii \mathscr{A}(\underline{U})Z,\quad Z(0)=\Tilde{\Lambda}_{h,0}^{\s}Y_0.   
\end{equation*}
Then we set $W=Z-\Tilde{\Lambda}^{\s}_{h,0}Y$ and $W$ satisfies 
\begin{equation*}
\partial_t W=\ii \mathscr{A}W+\ii [\mathscr{A},\Tilde{\Lambda}_{h,0}^{\s}]Y-(\partial_t\Tilde{\Lambda}_{h,0}^{\s})Y,
\end{equation*}
with initial data $W(0)=0$. With computations similar to the ones performed for the proof of \eqref{est-1}, we obtain
\begin{equation*}
    [\mathscr{A},\Tilde{\Lambda}_{h,0}^{\s}]-(\partial_t\Tilde{\Lambda}_{h,0}^{\s})=O(\epsilon_0)_{\mathscr{L}(\H^{\s},\H^{0})}.
\end{equation*}
As a consequence, by energy estimates,
\begin{align*}
\|\left(\mathcal{S}\Tilde{\Lambda}_{h,0}^{\s}-\Tilde{\Lambda}_{h,0}^{\s}\mathcal{S}\right)Y_0\|_{C([0,T],\H^{0})}&=\|W\|_{C([0,T],\H^{0})}\\
&\leq \|\ii [\mathscr{A},\Tilde{\Lambda}_{h,0}^{\s}]Y-(\partial_t\Tilde{\Lambda}_{h,0}^{\s})Y\|_{L^1([0,T],\H^{0})}\\
&\leq \|\ii [\mathscr{A},\Tilde{\Lambda}_{h,0}^{\s}]Y-(\partial_t\Tilde{\Lambda}_{h,0}^{\s})\|_{\mathscr{L}(\H^{\s},\H^{0})}\|Y\|_{L^1([0,T],\H^{\s})}\\
&\lesssim\epsilon_0\|Y\|_{C([0,T],\H^{\s})}\lesssim\epsilon_0\|Y_0\|_{\H^{\s}}.
\end{align*}

We conclude with the proof of \eqref{est-2}.
First, notice that $\chi_T^2$ commutes with $\Tilde{\Lambda}_{h,0}^{\s}$. It suffices to show that $[\varphi_{\omega}^2,\Tilde{\Lambda}_{h,0}^{\s}]$ is of order $\s-1$. We write $\varphi_{\omega}^2=\opbw(\varphi_{\omega}^2)+\left(\varphi_{\omega}^2-\opbw(\varphi_{\omega}^2)\right)$. Then by symbol calculus, we know that $[\opbw(\varphi_{\omega}^2),\Tilde{\Lambda}_{h,0}^{\s}]$ is of order $\s-1$, while $\varphi_{\omega}^2-\opbw(\varphi_{\omega}^2)$ is of order $-\infty$ since $\varphi_{\omega}^2$ is smooth.
This concludes the proof.
\end{proof}
We are in position to prove that the HUM operator defines an isomorphism on $\H_h^{\s}$.
\begin{proposition}\label{prop: HUM H^s isomorphism}
Under the hypothesis of Proposition \ref{prop: regularity of control operator}, for $h$ and $\epsilon_0$ sufficiently small, the HUM operator
\begin{equation*}
    \mathscr{K}|_{\H^{\s}_h}:\H^{\s}_h\rightarrow \H^{\s}_h
\end{equation*}
defines an isomorphism.
\end{proposition}
\begin{proof}
Consider the bilinear form 
 $   \alpha^h(U,V)=\poscals{\Lambda_{h,0}^{\s}\mathscr{K}U}{\Lambda_{h,0}^{\s}V}$ on $\H^{\s}_h$.
To obtain the coercivity, we consider $\forall U\in \H^{\s}_h$,
\begin{align*}
\alpha^h(U,U)&=\poscals{\Lambda_{h,0}^{\s}\mathscr{K}U}{\Lambda_{h,0}^{\s}U}\\
&=-\poscals{[\mathscr{K},\Lambda_{h,0}^{\s}]\Lambda_{h,0}^{-\s}\Lambda_{h,0}^{\s}U}{\Lambda_{h,0}^{\s}U}+\poscals{\mathscr{K}\Lambda_{h,0}^{\s}U}{\Lambda_{h,0}^{\s}U}.
\end{align*}
As a consequence of Proposition \ref{prop: HUM L^2-iosmorphism}, we obtain that 
$    \poscals{\mathscr{K}\Lambda_{h,0}^{\s}U}{\Lambda_{h,0}^{\s}U}\gtrsim \|\Lambda_{h,0}^{\s}U\|_{\H^{0}}.$
By Proposition \ref{lem: HUM commutator estimate}, we also know that
\begin{equation*}
    \poscals{[\mathscr{K},\Lambda_{h,0}^{\s}]\Lambda_{h,0}^{-\s}\Lambda_{h,0}^{\s}U}{\Lambda_{h,0}^{\s}U}\lesssim (\epsilon_0+h)\|\Lambda_{h,0}^{\s}U\|_{\H^0}.
\end{equation*}
Combining these two estimates, for $\epsilon_0$ and $h$ sufficiently small, we obtain that
\begin{equation*}
\alpha^h(U,U)\gtrsim  \|\Lambda_{h,0}^{\s}U\|_{\H^0}\gtrsim \|U\|_{\H^{\s}_h}.
\end{equation*}
Then we apply Lax-Milgram's theorem, we know that the HUM operator $\mathscr{K}|_{\H^{\s}_h}$ defines an isomorphism.
\end{proof}

\subsection{$H^s$-controllability}
We already showed the $L^2-$controllability of the simplified system
\begin{equation}\label{eq: simplified control system}
    \partial_tU=\ii \mathscr{A}(\underline{U})U-\ii \chi_T\varphi_{\omega}EF.
\end{equation}
Now we go back to the original paralinearized system
\begin{equation}\label{eq: para-control system }
    \partial_tU=\ii \mathscr{A}(\underline{U})U+R(\underline{U})U-\ii \chi_T\varphi_{\omega}EF
\end{equation}
We aim to prove the following proposition
\begin{proposition}\label{prop: construction of gHUM}
Suppose that $\omega$ satisfies the geometric control condition, $s_0>d/2+2$, $T>0$. Then for $\epsilon_0>0$ sufficiently small,  for $\underline{U}\in \mathscr{C}^{1,s_0}(T,\epsilon_0)$, there is an operator $\mathcal{L}_{P}=\mathcal{L}_{P}(\underline{U}): \H^{\s}\rightarrow C([0,T],\H^{\s})$, $\forall \s\geq s_0$ 
\begin{equation*}
  \|\mathcal{L}_P(\underline{U})W\|_{C([0,T],\mathcal{H}^{\s})}\lesssim \F(\|\underline{U}\|_{\H^{s_0}})\|W\|_{\H^{\s}},
\end{equation*}
for any $W\in C([0,T],\mathcal{H}^{\s})$ and where $\F\in C^0(\R;\R)$ is a non decreasing function equals to zero at the origin. This operator $\mathcal{L}_P$ null controls the system \eqref{eq: para-control system } for initial data $U(0)\in\H^{\s}$.
\end{proposition}
\begin{proof}
For the paralinearized control system \eqref{eq: para-control system }, we construct a new range operator $\mathcal{R}_{P}$.
For $G\in L^2([0,T],\H^{\s})$, let $Y\in C([0,T],\H^{\s})$ be the solution to the equation 
\begin{equation*}
\partial_tY=\ii \mathscr{A}(\underline{U})Y+R(\underline{U})Y+G,\quad Y(T)=0.
\end{equation*}
We set $\mathcal{R}_P(\underline{U})G=Y(0)$. Then we obtain a bounded linear operator from $L^2([0,T],\H^{\s})$ to $\H^{\s}$. This is a consequence of Corollary \ref{estimates-near-id} for the backward equation. Now we consider a linear equation
\begin{equation}\label{eq: new solution op}
\partial_tZ=\ii \mathscr{A}(\underline{U})Z-\ii \chi_T\varphi_{\o}E\mathscr{L}(Z_0),\quad Z(T)=0.
\end{equation}
Since $\mathscr{L}$ null controls \eqref{eq: simplified control system}, $Z(0)=Z_0$. We define a new solution operator $\mathcal{S}_P(Z_0)=Z\in C([0,T],\H^{\s})$, where $Z$ is a solution to \eqref{eq: new solution op} with $Z(0)=Z_0$, and $Z(T)=0$. Then $\mathcal{S}_P$ is a bounded linear operator from $\H^{\s}$ to $C([0,T],\H^{\s})$, which is also a consequence of Corollary \ref{estimates-near-id}. To construct the control operator $\mathcal{L}_P$, we start by perturbing $\mathscr{L}$. For $Z_0\in \H^{\s}$, set $Z=\mathcal{S}_P(Z_0)$ and $F=\mathscr{L}(Z_0)$. Let $Y$ be the solution to the equation 
\begin{equation*}
\partial_tY=\ii \mathscr{A}(\underline{U})Y+R(\underline{U})Y+R(\underline{U})Z,\quad Y(T)=0.    
\end{equation*}
Then we know that $W=Y+Z$ satisfies the equation
\begin{equation*}
\partial_t W=\ii \mathscr{A}(\underline{U})W+R(\underline{U})W-\ii \chi_T\varphi_{\omega}EF,\quad W(0)=Y|_{t=0}+Z_0, W(T)=0.    
\end{equation*}
By the definition of $\mathcal{R}_P$ and $\mathcal{S}_P$, 
\begin{equation*}
    Y|_{t=0}=\mathcal{R}_PR(\underline{U})Z=\mathcal{R}_PR(\underline{U})\mathcal{S}_PZ_0.
\end{equation*}
Define the perturbation operator $\mathcal{E}$ by
\begin{equation}\label{eq: perturbation op}
\mathcal{E}=\mathcal{R}_PR(\underline{U})\mathcal{S}_P: \H^{\s}\rightarrow \H^{\s} . 
\end{equation}
Using the perturbation operator $\mathcal{E}$, we obtain that
\begin{equation*}
    W(0)=(Id+\mathcal{E})Z_0.
\end{equation*}
By the estimate \eqref{R1}, we know that $\|\mathcal{E}\|_{\mathscr{L}(\H^{\s},\H^{\s})}\lesssim\epsilon_0$. Therefore, for $\epsilon_0$ sufficiently small, we obtain that 
$1+\mathcal{E}$ is an invertible operator from $\H^{\s}$ to itself. Then we set $\mathcal{L}_P=\mathscr{L}(Id+\mathcal{E})^{-1}$. The operator $\mathcal{L}_P$ is our desired control operator, which null controls 
the system \eqref{eq: para-control system }. 
\end{proof}

\subsection{Contraction Estimates}\label{sec: Contraction Estimates}
In this section, we prove some useful contraction estimates of several operators. In this section, we set $\underline{U}_1,\underline{U}_2\in \mathscr{C}^{1,s}(T,\epsilon_0)$. For simplicity, we use the following conventions. For any operator $\mathcal{P}=\mathcal{P}(\underline{U})$, we set $\mathcal{P}_i=\mathcal{P}(\underline{U}_i)$, $i=1,2$. For any symbols depending on $\underline{U}$, $a=a(\underline{U})$, we set $a_i=a(\underline{U}_i)$, $i=1,2$.
We first look at the range operator $\mathcal{R}$.
\begin{lemma}\label{lem: contraction for R}
Suppose that $s>\frac{d}{2}+2$, $\s\geq0$, then for $\epsilon_0$ sufficiently small,
\begin{equation*}
    \|\mathcal{R}_1-\mathcal{R}_2\|_{\mathscr{L}(L^2([0,T],\H^{\s+2}),\H^{\s})}\lesssim\|\underline{U}_1-\underline{U}_2\|_{L^{\infty}([0,T],\H^s)}.
\end{equation*}
\end{lemma}
\begin{proof}
For $G\in L^2([0,T],\H^{\s+2})$, let $Y_i\in C([0,T],\H^{\s+2})(i=1,2)$ be the solutions to the equations 
\begin{equation*}
    \partial_t Y_i=\ii \mathscr{A}(\underline{U}_i)Y_i+G,\quad Y_i(T)=0,i=1,2.
\end{equation*}
Then we set $Y=Y_1-Y_2$. $Y$ satisfies the equation
\begin{equation*}
    \partial_t Y=\ii \mathscr{A}(\underline{U}_1)Y+\ii(\mathscr{A}(\underline{U}_1)-\mathscr{A}(\underline{U}_2)) Y_2,\quad Y(T)=0.
\end{equation*}
By energy estimates, we have
\begin{equation*}
    \|Y_1(0)-Y_2(0)\|_{\H^{\s}}\lesssim \|(\mathscr{A}(\underline{U}_1)-\mathscr{A}(\underline{U}_2)) Y_2\|_{L^1([0,T],\H^{\s})}.
\end{equation*}
According to the estimate \eqref{Lip-simbolo}, we know that
\begin{align*}
    \|\left(\mathcal{R}_1-\mathcal{R}_2\right)G\|_{\H^{\s}}&\lesssim\|(\mathscr{A}(\underline{U}_1)-\mathscr{A}(\underline{U}_2)) Y_2\|_{L^1([0,T],\H^{\s})}\\
    &\lesssim \|\underline{U}_1-\underline{U}_2\|_{L^{\infty}([0,T],\H^s)}\|Y_2\|_{L^1([0,T],\H^{\s+2})}\\
    &\lesssim\|\underline{U}_1-\underline{U}_2\|_{L^{\infty}([0,T],\H^s)}\|G\|_{L^2([0,T],\H^{\s+2})}.
\end{align*}
\end{proof}
Then we look at the solution operator $\mathcal{S}$.
\begin{lemma}\label{lem: contraction for S}
Suppose that $s>\frac{d}{2}+2$, $\s\geq0$, then for $\epsilon_0$ sufficiently small,
\begin{equation*}
    \|\mathcal{S}_1-\mathcal{S}_2\|_{\mathscr{L}(\H^{\s+2},C([0,T],\H^{\s}))}\lesssim\|\underline{U}_1-\underline{U}_2\|_{L^{\infty}([0,T],\H^s)}.
\end{equation*}
\end{lemma}
\begin{proof}
For $Z_0\in \H^{\s+2}$, let $Z_i\in C([0,T],\H^{\s+2})(i=1,2)$ be the solutions to the equations 
\begin{equation*}
    \partial_t Z_i=\ii \mathscr{A}(\underline{U}_i)Z_i,\quad Z_i(0)=Z_0,i=1,2.
\end{equation*}
Then we set $Z=Z_1-Z_2$. $Z$ satisfies the equation
\begin{equation*}
    \partial_t Z=\ii \mathscr{A}(\underline{U}_1)Z+\ii(\mathscr{A}(\underline{U}_1)-\mathscr{A}(\underline{U}_2)) Z_2,\quad Z(0)=0.
\end{equation*}
By energy estimates, we have
\begin{equation*}
    \|Z_1-Z_2\|_{C([0,T],\H^{\s})}\lesssim \|(\mathscr{A}(\underline{U}_1)-\mathscr{A}(\underline{U}_2)) Z_2\|_{L^1([0,T],\H^{\s})}.
\end{equation*}
According to the estimate \eqref{Lip-simbolo}, we know that
\begin{align*}
    \|\left(\mathcal{S}_1-\mathcal{S}_2\right)Z_0\|_{\H^{\s}}&\lesssim\|(\mathscr{A}(\underline{U}_1)-\mathscr{A}(\underline{U}_2)) Z_2\|_{L^1([0,T],\H^{\s})}\\
    &\lesssim \|\underline{U}_1-\underline{U}_2\|_{L^{\infty}([0,T],\H^s)}\|Z_2\|_{L^1([0,T],\H^{\s+2})}\\
    &\lesssim\|\underline{U}_1-\underline{U}_2\|_{L^{\infty}([0,T],\H^s)}\|Z_0\|_{\H^{\s+2}}.
\end{align*}
\end{proof}
Based on the two lemmas above, we look at the HUM operator $\mathscr{K}$.
\begin{lemma}\label{lem: contraction for K}
Suppose that $s>\frac{d}{2}+2$, $\s\geq0$, then for $\epsilon_0$ sufficiently small,
\begin{equation*}
    \|\mathscr{K}_1-\mathscr{K}_2\|_{\mathscr{L}(\H^{\s+2},\H^{\s})}\lesssim\|\underline{U}_1-\underline{U}_2\|_{L^{\infty}([0,T],\H^s)}.
\end{equation*}
\end{lemma}
\begin{proof}
First recall the definition of $\mathscr{K}=-\mathcal{R}\chi_T^2\varphi_{\o}^2\mathcal{S}$.
\begin{align*}
    \mathscr{K}_1-\mathscr{K}_2&=-\mathcal{R}_1\chi_T^2\varphi_{\o}^2\mathcal{S}_1+\mathcal{R}_2\chi_T^2\varphi_{\o}^2\mathcal{S}_2\\
    &=(\mathcal{R}_2-\mathcal{R}_1)\chi_T^2\varphi_{\o}^2\mathcal{S}_1+\mathcal{R}_2\chi_T^2\varphi_{\o}^2(\mathcal{S}_2-\mathcal{S}_1).
\end{align*}
Therefore, by Lemma \ref{lem: contraction for R} and Lemma \ref{lem: contraction for S},
\begin{align*}
    \|\mathscr{K}_1-\mathscr{K}_2\|_{\mathscr{L}(\H^{\s+2},\H^{\s})}&\lesssim\|\mathcal{R}_1-\mathcal{R}_2\|_{\mathscr{L}(L^2([0,T],\H^{\s+2}),\H^{\s})}+\|\mathcal{S}_1-\mathcal{S}_2\|_{\mathscr{L}(\H^{\s+2},C([0,T],\H^{\s}))}\\
    &\lesssim \|\underline{U}_1-\underline{U}_2\|_{L^{\infty}([0,T],\H^s)}.
\end{align*}
\end{proof}
Now we look at the control operator $\mathscr{L}=\ii \chi_T\varphi_{\o}E\mathcal{S}\mathscr{K}^{-1}$.
\begin{lemma}\label{lem: contraction for L}
Suppose that $s>\frac{d}{2}+2$, $\s\geq0$, then for $\epsilon_0$ sufficiently small,
\begin{equation*}
    \|\mathscr{L}_1-\mathscr{L}_2\|_{\mathscr{L}(\H^{\s+2},C([0,T],\H^{\s}))}\lesssim\|\underline{U}_1-\underline{U}_2\|_{L^{\infty}([0,T],\H^s)}.
\end{equation*}
\end{lemma}
\begin{proof}
By the definition of $\mathscr{L}$, it suffices to estimate $\|\mathcal{S}_1\mathscr{K}_1^{-1}-\mathcal{S}_2\mathscr{K}_2^{-1}\|_{\mathscr{L}(\H^{\s+2},C([0,T],\H^{\s}))}$. Using the identity
\begin{equation*}
    \mathcal{S}_1\mathscr{K}_1^{-1}-\mathcal{S}_2\mathscr{K}_2^{-1}=(\mathcal{S}_1-\mathcal{S}_2)\mathscr{K}_1^{-1}-\mathcal{S}_2(\mathscr{K}_1^{-1}-\mathscr{K}_2^{-1}),
\end{equation*}
combining with the invertibility of $\mathscr{K}$ (Proposition \ref{prop: HUM L^2-iosmorphism}), we obtain 
\begin{equation*}
\|\mathcal{S}_1\mathscr{K}_1^{-1}-\mathcal{S}_2\mathscr{K}_2^{-1}\|_{\mathscr{L}(\H^{\s+2},C([0,T],\H^{\s}))}\lesssim\|\mathcal{S}_1-\mathcal{S}_2\|_{\mathscr{L}(\H^{\s+2},C([0,T],\H^{\s}))}+\|\mathscr{K}_1^{-1}-\mathscr{K}_2^{-1}\|_{\mathscr{L}(\H^{\s+2},\H^{\s})}. 
\end{equation*}
For the term $\mathscr{K}_1^{-1}-\mathscr{K}_2^{-1}$, we estimate by
\begin{align*}
    \|\mathscr{K}_1^{-1}-\mathscr{K}_2^{-1}\|_{\mathscr{L}(\H^{\s+2},\H^{\s})}&= \|\mathscr{K}_1^{-1}(\mathscr{K}_2-\mathscr{K}_1)\mathscr{K}_2^{-1}\|_{\mathscr{L}(\H^{\s+2},\H^{\s})}\\
    &\lesssim\|\mathscr{K}_2-\mathscr{K}_1\|_{\mathscr{L}(\H^{\s+2},\H^{\s})}.
\end{align*}
Therefore, we conclude by Lemma \ref{lem: contraction for S} and Lemma \ref{lem: contraction for K}.
\end{proof}
Then we analyze the operator $\mathcal{S}_P$.
\begin{lemma}\label{lem: contraction for S_P}
Suppose that $s>\frac{d}{2}+2$, $\s\geq0$, then for $\epsilon_0$ sufficiently small,
\begin{equation*}
    \|(\mathcal{S}_P)_1-(\mathcal{S}_P)_2\|_{\mathscr{L}(\H^{\s+2},C([0,T],\H^{\s}))}\lesssim\|\underline{U}_1-\underline{U}_2\|_{L^{\infty}([0,T],\H^s)}.
\end{equation*}
\end{lemma}
\begin{proof}
For $Z_0\in \H^{\s+2}$, let $Z_i=(\mathcal{S}_P)_i Z_0\in C([0,T],\H^{\s+2})(i=1,2)$ be the solutions to the equations 
\begin{equation*}
    \partial_t Z_i=\ii \mathscr{A}_i Z_i-\ii \chi_T\varphi_{\o}E\mathscr{L}_i(Z_0),\quad Z(T)=0,i=1,2.
\end{equation*}
Then we set $Z=Z_1-Z_2$. $Z$ satisfies the equation
\begin{equation*}
    \partial_t Z=\ii \mathscr{A}_1Z+\ii(\mathscr{A}_1-\mathscr{A}_2) Z_2+\ii \chi_T\varphi_{\o}E(\mathscr{L}_2-\mathscr{L}_1)(Z_0),\quad Z(0)=Z(T)=0.
\end{equation*}
By energy estimates, we have
\begin{equation*}
    \|Z_1-Z_2\|_{C([0,T],\H^{\s})}\lesssim \|(\mathscr{A}_1-\mathscr{A}_2) Z_2\|_{L^1([0,T],\H^{\s})}+\|(\mathscr{L}_2-\mathscr{L}_1)(Z_0)\|_{L^1([0,T],\H^{\s})}.
\end{equation*}
According to the estimate \eqref{Lip-simbolo} and Lemma \ref{lem: contraction for L}, we know that
\begin{align*}
    &\|\left((\mathcal{S}_P)_1-(\mathcal{S}_P)_2\right)Z_0\|_{\H^{\s}}\lesssim\|(\mathscr{A}_1-\mathscr{A}_2) Z_2\|_{L^1([0,T],\H^{\s})}+\|(\mathscr{L}_2-\mathscr{L}_1)(Z_0)\|_{L^1([0,T],\H^{\s})}\\
    &\lesssim \|\underline{U}_1-\underline{U}_2\|_{L^{\infty}([0,T],\H^s)}\|Z_2\|_{L^1([0,T],\H^{\s+2})}+\|\mathscr{L}_2-\mathscr{L}_1\|_{\mathscr{L}(\H^{\s+2},C([0,T],\H^{\s}))}\|Z_0\|_{\H^{\s+2}}\\
    &\lesssim \|\underline{U}_1-\underline{U}_2\|_{L^{\infty}([0,T],\H^s)}\|Z_0\|_{\H^{\s+2}}.
\end{align*}
\end{proof}
Then we look at the refined range operator $\mathcal{R}_P$.
\begin{lemma}\label{lem: contraction for R_P}
Suppose that $s>\frac{d}{2}+2$, then for $\epsilon_0$ sufficiently small,
\begin{equation*}
    \|(\mathcal{R}_P)_1-(\mathcal{R}_P)_2\|_{\mathscr{L}(L^2([0,T],\H^{s}),\H^{s-2})}\lesssim\|\underline{U}_1-\underline{U}_2\|_{L^{\infty}([0,T],\H^{s-2})}.
\end{equation*}
\end{lemma}
\begin{proof}
For $G\in L^2([0,T],\H^{s})$, let $Y_i\in C([0,T],\H^{s})(i=1,2)$ be the solutions to the equations 
\begin{equation*}
    \partial_tY_i=\ii \mathscr{A}_iY_i+R_iY_i+G,\quad Y_i(T)=0.,i=1,2.
\end{equation*}
Then we set $Y=Y_1-Y_2$. $Y$ satisfies the equation
\begin{equation*}
    \partial_tY=\ii \mathscr{A}_1Y+R_1Y+\ii(\mathscr{A}_1-\mathscr{A}_2) Y_2+(R_1-R_2)Y_2,\quad Y(T)=0.
\end{equation*}
By energy estimates, we have
\begin{equation*}
    \|Y_1(0)-Y_2(0)\|_{\H^{s-2}}\lesssim \|(\mathscr{A}_1-\mathscr{A}_2) Y_2\|_{L^1([0,T],\H^{s-2})}+\|(R_2-R_1)Y_2\|_{L^1([0,T],\H^{s-2})}.
\end{equation*}
According to the estimate \eqref{Lip-simbolo} and \eqref{R3}, we know that
\begin{align*}
    \|\left((\mathcal{R}_P)_1-(\mathcal{R}_P)_2\right)G\|_{\H^{s-2}}&\lesssim\|(\mathscr{A}_1-\mathscr{A}_2) Y_2\|_{L^1([0,T],\H^{s-2})}+\|(R_2-R_1)Y_2\|_{L^1([0,T],\H^{s-2})}\\
    &\lesssim \left(\|\underline{U}_1-\underline{U}_2\|_{L^{\infty}([0,T],\H^{s-2})}+\|R_2-R_1\|_{\mathscr{L}(\H^{s},\H^{s})}\right)\|Y_2\|_{L^1([0,T],\H^{s})}\\
    &\lesssim \|\underline{U}_1-\underline{U}_2\|_{L^{\infty}([0,T],\H^{s-2})}\|G\|_{L^2([0,T],\H^{s})}.
\end{align*}
\end{proof}
Then we consider the perturbation operator $\mathcal{E}$.
\begin{lemma}\label{lem: contraction for E}
Suppose that $s>\frac{d}{2}+2$, then for $\epsilon_0$ sufficiently small,
\begin{equation*}
    \|\mathcal{E}_1-\mathcal{E}_2\|_{\mathscr{L}(\H^{s},H^{s-2})}\lesssim\|\underline{U}_1-\underline{U}_2\|_{L^{\infty}([0,T],\H^{s-2})}.
\end{equation*}
\end{lemma}
\begin{proof}
We use the identity
\begin{align*}
\mathcal{E}_1-\mathcal{E}_2&=(\mathcal{R}_P)_1R_1(\mathcal{S}_P)_1- (\mathcal{R}_P)_2R_2(\mathcal{S}_P)_2\\
&=\left((\mathcal{R}_P)_1-(\mathcal{R}_P)_2\right)R_1(\mathcal{S}_P)_1+ (\mathcal{R}_P)_2(R_1-R_2)(\mathcal{S}_P)_1+(\mathcal{R}_P)_2R_2\left((\mathcal{S}_P)_1-(\mathcal{S}_P)_2\right).
\end{align*}
We conclude by Lemma \ref{lem: contraction for R_P}, Lemma \ref{lem: contraction for S_P} and \eqref{R3}.
\end{proof}
Finally, we consider the control operator $\mathcal{L}_P$.
\begin{lemma}\label{lem: contraction for L_P}
Suppose that $s>\frac{d}{2}+2$, then for $\epsilon_0$ sufficiently small,
\begin{equation}\label{Lip-control}
    \|(\mathcal{L}_P)_1-(\mathcal{L}_P)_2\|_{\mathscr{L}(\H^{s},C([0,T],H^{s-2}))}\lesssim\|\underline{U}_1-\underline{U}_2\|_{L^{\infty}([0,T],\H^{s-2})}.
\end{equation}
\end{lemma}
\begin{proof}
We use the identity
\begin{align*}
    (\mathcal{L}_P)_1-(\mathcal{L}_P)_2&=\mathscr{L}_1(Id+\mathcal{E}_1)^{-1}-\mathscr{L}_2(Id+\mathcal{E}_2)^{-1}\\
    &=(\mathscr{L}_1-\mathscr{L}_2)(Id+\mathcal{E}_1)^{-1}+\mathscr{L}_2(Id+\mathcal{E}_1)^{-1}(\mathcal{E}_2-\mathcal{E}_1)(Id+\mathcal{E}_2)^{-1}.
\end{align*}
We conclude by Lemma \ref{lem: contraction for L} and Lemma \ref{lem: contraction for E}.
\end{proof}

\section{Nonlinear null controllability}\label{sec: nonlinear control}
In section \ref{sec: L^2 controllability for the linear system}, we introduced an iterative scheme to construct a solution for the nonlinear control problem \eqref{eq: nonlinear control problem}. Suppose that $s\geq s_0>\frac{d}{2}+2$ is sufficiently large, fix $0<\epsilon_0<1$ sufficiently small, such that the estimates in previous sections hold. In this section, we need to be very careful with the constants. Let $C>10$ be a constant such that 
\begin{enumerate}
    \item For all $\underline{U}\in \mathscr{C}^{1,s}(T,\epsilon_0)$ and $\s\in[0,s]$, 
    \begin{equation*}
    \|\mathcal{L}_P(\underline{U})\|_{\mathscr{L}(\H^{\s},C([0,T],\mathcal{H}^{\s}))}+  \|\mathscr{A}(\underline{U})\|_{L^{\infty}([0,T],\mathscr{L}(\H^{\s},\H^{\s-2}))}\leq C.
    \end{equation*}
    \item For $\s\in[s-2,s]$, $G\in L^{\infty}([0,T],\H^{\s})$, if $Y$ satisfies the equation
    \begin{equation*}
     \partial_tY=\ii \mathscr{A}(\underline{U})Y+R(\underline{U})Y+G,\quad Y(T)=0\text{ or }Y(0)=0,   
    \end{equation*}
    then we have the energy estimate
    \begin{equation}\label{eq: uniform bound energy estimate}
     \|Y\|_{C([0,T],\H^{\s})}+ \|\partial_t Y\|_{L^{\infty}([0,T],\H^{\s-2})} \leq C\|G\|_{L^{\infty}([0,T],\H^{\s})}. 
    \end{equation}
    \item In all lemmas and propositions of Section \ref{sec: Contraction Estimates}, the statements remain valid after replacing l.h.s $\lesssim$ r.h.s. by l.h.s. $\leq C\times$r.h.s. 
\end{enumerate}
Let $U_{in}\in \H^s$ with $s\geq s_0>\frac{d}{2}+2$ be such $\|U_{in}\|\leq \frac{\epsilon_0}{C^8}$. We define a sequence  $(U^n,F^n)$ by the problems \eqref{n-prob} (with $(U^0,F^0)=(0,0)$). In this section we aim to prove that $\{F^n\}$ converges in $C^0([0,T],\H^s)$. Moreover we prove that the limit null controls the quasilinear Schr\"odinger equation \eqref{NLS}. For any operator $P(U^n)$ depending on $U^n$, for simplicity, we denote by $P_n=P(U^n)$. Recall the definition of
    $F^{n+1}=\mathcal{L}_P(U^{n})U_{in}\in C([0,T],\H^s)$ and that 
$U^{n+1}\in C([0,T],\H^s)$ is the solution to  \eqref{n-prob}.
\begin{lemma}\label{iteration}
Consider \eqref{n-prob} and let $s_0>\frac{d}{2}+2$. If $\| U_{in}\|_{\H^s}\leq \frac{\epsilon_0}{C^8}$ is sufficiently small, then for any $s\geq s_0$, we have the following estimates
\begin{align}
&U^n\in \mathscr{C}^{1,s}(T,\epsilon_0)\label{un-bdd}\\
&U^{n+1}-U^n\in \mathscr{C}^{1,s-2}(T,\epsilon^n_0)\label{un-conv}\\
&F^n\in \mathscr{C}^{0,s}(T,\epsilon_0)\label{fn-bdd}\\
&F^{n+1}-F^n\in \mathscr{C}^{0,s-2}(T,\epsilon^n_0) \label{fn-conv}.
\end{align}
\end{lemma}
\begin{proof}
We proceed by induction, the first step is trivial. We suppose the statement true at rank $n$ and we prove it at rank $n+1$.
We use the notation introduced below \eqref{eq: simplified control system-L^2}. Let us prove \eqref{un-bdd} at rank $n+1$. The \eqref{un-bdd} is a consequence of the energy estimates by Corollary \ref{estimates-near-id}, indeed we have
\begin{equation}\label{eq: U^n+1-estimate}
\begin{aligned}
||U^{n+1}||_{C([0,T];\H^s)}+||\partial_t U^{n+1}||_{L^{\infty}([0,T]; {\H}^{s-2})}&\leq C\|F^{n+1}\|_{L^{\infty}([0,T],\H^s)}  \\
            &\leq C\|(\mathcal{L}_P)_{n}\|_{\mathscr{L}(\H^{s},C([0,T],\mathcal{H}^{s}))}\|U_{in}\|_{\H^s}\\
            &\leq C^2\times \frac{\epsilon_0}{C^8}=\frac{\epsilon_0}{C^6}.
\end{aligned}
\end{equation}
This gives us that $U^{n+1}\in \mathscr{C}^{1,s}(T,\epsilon_0)$ by $C>10$. Concerning the \eqref{fn-bdd} at rank $n+1$, one has to use Proposition \ref{prop: construction of gHUM},
\begin{align*}
||F^{n+1}||_{C([0,T];\H^s)}&=||(\mathcal{L}_P)_{n}U_{in}||_{C([0,T];\H^s)}\\
&\leq C\|U_{in}\|_{\H^s}\leq C\times \frac{\epsilon_0}{C^8}=\frac{\epsilon_0}{C^7}.
\end{align*}
This implies that $F^{n+1}\in \mathscr{C}^{0,s}(T,\epsilon_0)$. We prove \eqref{fn-conv} at rank $n+1$. By Proposition \ref{prop: construction of gHUM} we have
\begin{equation}\label{eq:F^n+1-conv}
\begin{aligned}
\|F^{n+2}-F^{n+1}\|_{C([0,T];\H^{s-2})}&=\|((\mathcal{L}_P)_{n+1}-(\mathcal{L}_P)_{n})U_{in}\|_{\H^{s-2}}\\
&\stackrel{\eqref{Lip-control}}{\leq} C \|U^{n+1}-U^{n}\|_{\H^{s-2}}\|U_{in}\|_{\H^{s}}\\
&\stackrel{\eqref{un-conv}}{\leq} C\epsilon_0^n\times\frac{\epsilon_0}{C^8}=\frac{\epsilon_0^{n+1}}{C^7},
\end{aligned}
\end{equation}
one concludes by noticing that $C>10$.\\
Let us prove \eqref{un-conv} at rank $n+1$. We set $W^{n+1}:=U^{n+2}-U^{n+1}$, then we have 
\begin{equation}\label{prova-un-conv}
\begin{aligned}
\partial_t W^{n+1}&-\ii \A_{n+1}W^{n+1}-R_{n+1}W^{n+1}\\
&+\ii(\A_n-\A_{n+1})U^{n+1}+(R_n-R_{n+1})U^{n+1}\\
&+\ii\chi_T\varphi_{\omega}E(F^{n+2}-F^{n+1})=0,
\end{aligned}
\end{equation}
with $W^{n+1}(0)=0$.
We apply Corollary \ref{estimates-near-id} to the equation above with $F=\ii(\A_n-\A_{n+1})U^{n+1}+(R_n-R_{n+1})U^{n+1}
+\ii\chi_T\varphi_{\omega}E(F^{n+2}-F^{n+1})$ and we get
\begin{align*}
||W^{n+1}||_{C([0,T];\H^{s-2})}+||\partial_t W^{n+1}||_{L^{\infty}([0,T]; {\H}^{s-4})}&\leq C\|\A_{n+1}-\A_n)U^{n+1}\|_{L^{\infty}([0,T],\H^{s-2})}\\\
&+C\|(R_{n+1}-R_n))U^{n+1}\|_{L^{\infty}([0,T],\H^{s-2})}\\
&+C\|F^{n+2}-F^{n+1}\|_{L^{\infty}([0,T],\H^{s-2})}.
\end{align*}
By using \eqref{Lip-simbolo} and \eqref{R3}, we obtain
\begin{align*}
\|(\A_{n+1}-\A_{n})U^{n+1}\|_{L^{\infty}([0,T],\H^{s-2})}&\leq C \|U^{n+1}-U^{n}\|_{\H^{s-2}}\|U^{n+1}\|_{\H^s} ,\\
\|(R_{n+1}-R_n))U^{n+1}\|_{L^{\infty}([0,T],\H^{s-2})}&\leq C \|U^{n+1}-U^{n}\|_{\H^{s-2}}\|U^{n+1}\|_{\H^s}.
\end{align*}
By \eqref{eq: U^n+1-estimate}, \eqref{un-conv} at rank $n$, and \eqref{eq:F^n+1-conv},  
\begin{align*}
||W^{n+1}||_{C([0,T];\H^{s-2})}+||\partial_t W^{n+1}||_{L^{\infty}([0,T]; {\H}^{s-4})}&\leq 2C^2\|U^{n+1}-U^{n}\|_{\H^{s-2}}\|U^{n+1}\|_{\H^s}+\frac{\epsilon_0^{n+1}}{C^7}\\
&\leq  2C^2\epsilon_0^{n}\times\frac{\epsilon_0}{C^7}+\frac{\epsilon_0^{n+1}}{C^7}\leq\epsilon_0^{n+1}.
\end{align*}
\end{proof}

We are now in position to prove Theorem \ref{main-null}.
\begin{proof}[Proof of Theorem \ref{main-null}]
By Lemma \ref{iteration} we deduce that $U^n$ and $F^n$ are Cauchy sequences in $C^0([0,T];\H^{s-2})$, therefore they converge to some limites $U$ and $F$ in $C^0([0,T];\H^{s-2})$ respectively. Using interpolation and \eqref{un-bdd}, \eqref{fn-bdd} we may prove that $U^n$ and $F^n$ strongly converge in $C^0([0,T];\H^{s'})$, for any $s'<s$. For this reason $U$ and $F$ belong to $C^0([0,T];\H^{s'}) \cap L^{\infty}([0,T];\H^s)$. Passing to the limit in \eqref{n-prob}, we deduce that $U$ and $F$, of the form \eqref{matriciozze}, satisfy \eqref{NLSpara} so that $u$ and $f$ satisfy \eqref{NLS} with $u(T)=0$.
 It remains to prove that  $U$ and $F$ are in  $C^0([0,T];\H^{s})$. We use the Bona-Smith technique, we define $U_0^N:=\sum_{|k|\leq N}\hat{U}_{0,k}e^{\ii k\cdot x}$, where $U_0$ is the initial condition of \eqref{NLSpara} and $\hat{U}_{0,k}$ is the $k$-th Fourier coefficient of $U_0$. Since $U_0^N$ is $C^{\infty}$, then the solution $U^N$ of \eqref{NLSpara} belongs to $C^0([0,T];\H^s)$ for any $s> d/2+2$, the same holds true for $F^N$, which is the control for the equation with initial condition $U_0^N$. We immediately note that, thanks to \eqref{Lip-control}, we have 
 \begin{equation}\label{pupi}
 \|F-F^N\|_{C([0,T],\H^s)}\leq C \|U^N-U\|_{L^{\infty}([0,T],\H^{s})}(\|U_0^N\|_{\H^s}+\|U_0\|_{\H^s}).
 \end{equation}
 Let $W^N=U-U^N$ and let $\sigma=s-2-\varepsilon$. The function $W^N$ solves the equation 
 \begin{align*}
 \partial_t W^N&+\A(U)W^N+R(U)W^N\\
 &+\ii(\A(U^N)-\A(U))U^N +(R(U^N)-R(U))U^N=F^N-F.
 \end{align*}
 We use Corollary \ref{estimates-near-id} to establish the existence of the solution of the problem $ \partial_t W^N+\A(U)W^N+R(U)W^N=F^N-F$, then by Duhamel formula, the estimate on $\|F-F^N\|_{\H^s}$ in \eqref{pupi}, \eqref{Lip-simbolo}, \eqref{R3} we obtain 
 \begin{equation*}
 \|W^N\|_{\H^{\s}}\leq \|U_0-U_0^N\|_{\H^{\s}}+C\int_0^t\|W^N\|_{\H^{\s}}(\|U\|_{\H^{\s}}+\|U^N\|_{\H^{\s+2}}).    
 \end{equation*}
 The sequence $U^N$ is uniformly bounded in $\H^{\s+2}$, since $\s+2<s$. By Gronwall  inequality we conclude that $\|W^N\|_{\H^{\s}}\leq \|U_0-U_0^N\|_{\H^{\s}}$, which by using smoothing properties of the projector may be bounded by $\|W^N\|_{\H^{\s}}\leq N^{-2-\varepsilon}\|U_0\|_{\H^s}$. We are ready to prove the convergence in the high norm $\H^s$. Doing a similar computation we may prove that $\|W^N\|_{\H^s}\leq \|U_0^N-U_0\|_{\H^s}+\|W^N\|_{H^{\s}}\|U_0^N\|_{H^{s+2}}$.
 Since $\|U_0^N\|_{H^{s+2}}\leq N^2\|U_0\|_{\H^s}$ we conclude by using the estimate on $\|W^N\|_{\H^{\s}}$. As a consequence of \eqref{pupi} we also get that $F^N-F$ strongly converges to $0$ in $C^0([0,T];\H^s)$.
\end{proof}

\section{Exact controllability}\label{sec: establish the exact controllability}
In this section, we deduce Theorem \ref{main} from Theorem \ref{main-null}. We first consider the time reversed equation (that is, the equation obtained by the change of variable $t\mapsto-t$) of \eqref{eq: nonlinear control problem} as follows:
\begin{equation}\label{eq: time reversed nonlinear eq}
 -\partial_t V=-\ii\A(V) V+R(V)V-\ii\chi_T\varphi_{\omega}EF_V  
\end{equation}
After repeating the procedure presented in Section \ref{sec: L^2 controllability for the linear system}, Section \ref{sec: H^s linear control} and Section \ref{sec: nonlinear control}, we could prove that the null controllability holds the time reversed system \eqref{eq: time reversed nonlinear eq} as well, with the same proof. Now we choose $u_{in},u_{end}\in H^s(\T^d,\C)$, satisfying that
\begin{equation*}
 \|u_{in}\|_{H^{s}}+\|u_{end}\|_{H^{s}}<\epsilon_0,   
\end{equation*}
with $\epsilon_0$ sufficiently small. We construct states $U_{in}$ and $U_{end}$ by
\begin{equation}
    U_{in}=\vect{u_{in}}{\bar{u}_{in}},\quad U_{end}=\vect{u_{end}}{\bar{u}_{end}}.
\end{equation}
Then there exists $F_W\in C([0,\frac{T}{2}],\H^s)$ which null controls the system \eqref{eq: nonlinear control problem}, such that
\begin{equation}
\left\{
\begin{array}{l}
     \partial_t W=-\ii\A(W) W+R(W)W-\ii\chi_{\frac{T}{2}}\varphi_{\omega}EF_W, \\
     W|_{t=0}=U_{in},\\
     W|_{t=\frac{T}{2}}=0.
\end{array}
\right.
\end{equation}
Similarly, there exists $F_V\in C([0,\frac{T}{2}],\H^s)$ which null controls the system \eqref{eq: time reversed nonlinear eq}, such that
\begin{equation}
\left\{
\begin{array}{l}
     -\partial_t V=-\ii\A(V) W+R(V)V-\ii\chi_{\frac{T}{2}}\varphi_{\omega}EF_V, \\
     V|_{t=0}=U_{end},\\
     V|_{t=\frac{T}{2}}=0.
\end{array}
\right.
\end{equation}
Moreover, we know that the solutions $W,V\in C([0,\frac{T}{2}],\H^s)$. Now we define $U\in C([0,T],\H^s)$ and $F\in C([0,T],\H^s)$ by
\begin{equation*}
    U(t)=\left\{
    \begin{array}{ll}
         W(t)&t\in[0,\frac{T}{2}],  \\
         V(T-t)&t\in(\frac{T}{2},T].
    \end{array}
    \right.\quad \text{ and }
    F(t)=\left\{
    \begin{array}{ll}
         \chi_{\frac{T}{2}}(t)F_W(t)&t\in[0,\frac{T}{2}],  \\
         \chi_{\frac{T}{2}}(T-t)F_V(T-t)&t\in(\frac{T}{2},T].
    \end{array}
    \right.
\end{equation*}
Indeed, $F$ continues in time, since the cut-off function vanishes near $t=\frac{T}{2}$. $U$ solves the equation:
\begin{equation}
 \left\{
\begin{array}{l}
     \partial_t U=-\ii\A(U) U+R(U)U-\ii\varphi_{\omega}EF, \\
     U|_{t=0}=W|_{t=0}=U_{in},\\
     U|_{t=T}=V(T-t)|_{t=T}=V|_{t=0}=U_{end}.
\end{array}
\right.   
\end{equation}
This completes the proof of Theorem \ref{main}.

\def\cprime{$'$}



\end{document}